\documentclass[a4paper,11pt,twoside,reqno]{amsart}
\usepackage[utf8]{inputenc}
\usepackage[plainpages=false,pdfpagelabels=true]{hyperref}
\usepackage{amssymb,amsthm}
\usepackage[margin=1in]{geometry}
\usepackage{slashed}
\usepackage{graphicx}

\newtheorem{Satz}{Theorem}[section]
\newtheorem{Prop}[Satz]{Proposition}
\newtheorem{Lem}[Satz]{Lemma}
\newtheorem{Thm}[Satz]{Theorem}

\theoremstyle{definition}

\newtheorem{Bem}[Satz]{Remark}

\newtheorem{claim}{Claim}

\newcommand{\sff}{\mathrm{I\!I}}

\newcommand{\C}{{\mathbb{C}}}

\allowdisplaybreaks[1]

\renewcommand{\epsilon}{\varepsilon}

\newcommand{\R}{\ensuremath{\mathbb{R}}}

\newcommand{\ch}{\text{ch}}
\newcommand{\spec}{\text{spec}}

\newcommand{\lra}{\longrightarrow}

\newcommand{\rquot}[2]{\raisebox{0.5ex}{$#1$}\!/\!\raisebox{-0.5ex}{$#2$}}

\numberwithin{equation}{section}

\title{Dirac-harmonic maps from manifolds with boundary via index theory}

\date{\today}

\author{Volker Branding}
\address{University of Rostock, Institute of Mathematics\\
Ulmenstraße 69, 18057 Rostock, Germany}
\email{volker.branding@uni-rostock.de}

\author{Nicolas Ginoux}
\address{Universit\'{e} de Lorraine, CNRS, IECL, F-57000 Metz, France}
\email{nicolas.ginoux@univ-lorraine.fr}


\usepackage{color}

\begin{document}

\begin{abstract}
Dirac-harmonic maps are a mathematical version of the supersymmetric non-linear sigma model
of quantum field theory. 
Since they arise as critical points of an unbounded energy functional, it is a challenging task to establish general existence results.
Within this manuscript we mainly prove a general existence result based on index theory for manifolds with boundary and apply it in various situations. 
Due to the presence of a boundary there are additional contributions in the Atiyah-Patodi-Singer index formula allowing for a non-vanishing index, which is the key argument in our proofs.
Further explicit examples based on twistor spinors on surfaces are also presented.
\end{abstract} 

\maketitle

\section{Introduction and results}\label{s:introandresults}
The action functionals of quantum field theory  
provide interesting challenges for mathematicians.
Often, they are closely connected to geometric variational problems with a long tradition in differential geometry and geometric analysis. A prototype of such a functional is the nonlinear sigma model of quantum field theory corresponding to the classic energy functional in mathematics.
Here, one considers two Riemannian manifolds \((M,g),(N,h)\) and a map \(f\colon M\to N\).
Then, one is interested in finding the critical points of 
\begin{align}
\label{eq:energy}
    E(f)=\frac{1}{2}\int_M|df|^2\,d\mu_g.
\end{align}
The Euler-Lagrange equation of the above energy is a second order non-linear equation which makes its analysis involved from a technical point of view. In the mathematics literature the critical points of the energy
are precisely \emph{harmonic maps}, for an overview on the latter we refer to the book
\cite{MR2044031}.\\

In modern theoretical physics there are various extensions of the nonlinear sigma model that have attracted considerable attention over the years.
These extensions usually involve additional variables (or fields in the terminology of physics).
One such extension is obtained by augmenting the energy of a map with a spinor field.
This leads to the notion of \emph{Dirac-harmonic maps} which have been introduced in \cite{MR2262709}
as a mathematical version of the supersymmetric non-linear sigma model of quantum field theory. In contrast to the model studied by physicists, the notion of Dirac-harmonic maps does not employ
anti-commuting spinor fields and this approach allows to use well-developed tools from the geometric calculus of variations in the analysis of action functional from physics.\\

In order to set up the energy functional for Dirac-harmonic maps we need to assume that the manifold \(M\) is spin ensuring the existence of the spinor bundle \(\Sigma M\). 
We may then consider the twisted bundle 
\(\Sigma M\otimes f^\ast TN\) and sections \(\phi\in\Gamma(\Sigma M\otimes f^\ast TN)\) are called
\emph{vector spinors} on which the twisted Dirac operator \(D^f\colon \Gamma(\Sigma M\otimes f^\ast TN)\circlearrowleft\)
acts. 
For more details on spin geometry we refer to the books \cite{MR3410545}, \cite{MR1476425}, \cite{MR2509837} and \cite{MR1031992}.

Using this setup the energy functional for Dirac-harmonic maps is given by 
\begin{align}
E(f,\phi)=\frac{1}{2}\int_M|df|^2+\Re\langle\phi,D^f\phi\rangle\,d\mu_g.
\end{align}
Due to the presence of the Dirac term the energy functional \(E(f,\phi)\) is unbounded in both directions, so that most techniques employed to prove general existence results for harmonic maps can no longer be applied, despite the existence of a few explicit solutions of the Dirac-harmonic-map system on particular closed manifolds \cite{zbMATH07075240,MR2569270}.\\

In \cite{MR3070562} Ammann and the second author used the following strategy to reach existence results for Dirac-harmonic maps from closed manifolds. 
Assuming the existence of a harmonic map they employed the index theorem to construct a vector spinor such that the resulting pair is a Dirac-harmonic map. 
The Dirac-harmonic maps obtained via this approach are uncoupled in the sense that all terms in the Euler-Lagrange equations vanish independently. Similar existence results for the Dirac-Yang-Mills system were obtained in \cite{adam}.\\

In the case that the domain manifold is an expanding spacetime a global existence result for the (hyperbolic) Dirac-wave map system was achieved in \cite{MR3830277}. \\

The aim of this manuscript is to first conduct a thorough investigation of the boundary value problem
for Dirac-harmonic maps and to analyze the suitable elliptic boundary conditions.
To this end we first recall a number of key results on the existence of harmonic maps from manifolds with boundary.
First, we want to mention the seminal contribution of
Hamilton \cite[Theorems pp. 6-7]{MR482822} who showed the existence of a harmonic map in every homotopy class of maps provided
that the target manifold has non-positive curvature extending the famous existence result of Eells and Sampson.
In addition, Lemaire \cite[Theorems  1.2 \& 1.4]{MR664104} and Chang \cite[Theorem 1.1]{MR1030856} established existence results
for harmonic maps from surfaces with boundary.\\

Before we describe our main results, we give a non-exhaustive list of references that are connected to the boundary value problem for Dirac-harmonic maps, most of which focus on analytic aspects of the problem.
\begin{enumerate}
    \item The first reference on the boundary value problem for Dirac-harmonic maps is provided by \cite{MR3085099} where, besides a first analysis of the appropriate boundary conditions for the problem, regularity results for the boundary value problem of Dirac-harmonic maps are established. The above results were later refined in \cite{MR3908762}.
\item In the physics literature the supersymmetric sigma model for a domain manifold with boundary
is investigated by \cite{MR1962116}, \cite{MR2022994}. Here, the boundary is interpreted as a D-brane
for the superstring whose equation of motion is governed by the action of the supersymmetric sigma model,
see also \cite[Section 2.4.4]{zbMATH05557123}.
\item An existence result via the heat flow for Dirac-harmonic maps from surfaces with boundary was established in \cite{MR3724759} and \cite{MR4402492}. 
It turns out that the heat flow for Dirac-harmonic maps is substantially easier to handle in the case of a manifold with boundary. In the closed case only short-time existence could be achieved up to now \cite{MR3719555}.
\item A survey on heat-flow approaches to the boundary value problem for Dirac-harmonic maps is provided by \cite{MR4287930}.
\item A uniqueness result for Dirac-harmonic maps from a compact surface with boundary
was established in \cite{MR4549002}.
\item Questions concerning the boundary regularity of Dirac-harmonic maps from manifolds with boundary have recently been studied in \cite{MR4868118}, \cite{MR5102184}.
\end{enumerate}

In a series of articles \cite{MR397797,MR397798,MR397799} Atiyah, Patodi and Singer developed the index theory for Dirac operators on manifolds with boundary, introducing the famous boundary condition named after them.
Their index formula, which is recalled below, will be of great importance within this manuscript.
Further boundary conditions for Dirac operators have been studied since, see e.g. \cite{zbMATH06748684} for the treatment of very general elliptic boundary conditions.\\

In the closed case, the Atiyah-Singer index theorem relates the index of some Dirac-type operator with the so-called $\alpha$-genus of the manifold.
As a consequence, a non-vanishing $\alpha$-genus yields the existence of a non-trivial harmonic (twisted) spinor. 
For a manifold with boundary equipped with Atiyah–Patodi–Singer boundary conditions, the index is instead given by the APS formula.
In contrast to the closed case, the index is not determined solely by the interior characteristic form: it also depends on the spectral invariants of the induced boundary Dirac operator. 
Consequently, a non-vanishing \(\alpha\)-genus alone does not imply the kernel of the APS Dirac operator to be non-trivial.\\

However, in several cases we can get existence results for Dirac-harmonic maps exploiting the presence of a nonempty boundary.
Some of these results do not have a counterpart in the case of a closed manifold which suggests that the boundary value problem for Dirac-harmonic maps might be more attractive to study than the case of a closed domain manifold.
For instance, Theorem \ref{t:existDhmapsgAPSasmall} establishes a general existence result based on the spectral flow, while Theorem \ref{t:exDhMgr>=2} exploits the boundary contribution in the Atiyah-Singer index formula to reach a very general existence result on surfaces.\\

This article is organized as follows. 
In Section \ref{s:prelim} we recall background material on both harmonic and Dirac-harmonic maps, elliptic boundary conditions and the Atiyah-Patodi-Singer index theorem for manifolds with non-empty boundary. 
Section \ref{s:variationalformulas} presents the variational formula required for our analysis which leads to a first existence result for Dirac-harmonic maps from manifolds with boundary, see Propositions \ref{p:existuncoupledDhfromvarformula} and \ref{p:examplesDhmapsZ2grading}.
Section \ref{s:firstexamplesviatwistorspinors} presents explicit solutions of the equations for Dirac-harmonic maps based on an explicit ansatz involving twistor spinors, see Propositions \ref{p:explicitexdim2chiral} and \ref{p:explicitexdim2MIT}.
Finally, in Section \ref{s:solfromindexthm}, we establish various existence results for Dirac-harmonic maps from manifolds with boundary based on the index theorem, see e.g. Theorems \ref{t:exDhMgr>=2} when $\dim(M)=2$ and \ref{t:exdim4} when $\dim(M)=4$.

\medskip
\textbf{Acknowledgements}: The first-named author benefited from ``trajectoires'' funds from the Institut Elie Cartan de Lorraine, which he would like to thank.
Part of this work was also carried out at University of Rostock, which the second-named author would like to thank for its hospitality.
The authors would also like to thank Victor Nistor for 
inspiring discussions, in particular the idea leading to Theorem \ref{t:existDhmapsgAPSasmall} below.

\section{Preliminary Results}\label{s:prelim}
In this section, we review several well-established results from the literature that will serve as important tools in proving our main results.
This includes existence results for harmonic maps from manifolds with boundary, a thorough review of elliptic boundary conditions for twisted Dirac operators and a brief introduction to the Atiyah-Singer index theorem for manifolds with boundary.

\subsection{Harmonic maps from manifolds with boundary}
First, we present a number of existence results for harmonic maps from Riemannian manifolds with non-empty boundary \(\partial M\). Recall that a map \(f\colon M\to N\) is called harmonic if it is a critical point of the energy \eqref{eq:energy} which is equivalent to the map having vanishing tension field 
\(\tau(f):=\mathrm{tr}_g(\nabla df)\).

We have the following existence result due to Hamilton \cite[Theorems pp. 6-7]{MR482822}:

\begin{Satz}\label{t:Hamilton}
Let \((M,g)\) be a compact Riemannian manifold with boundary and \((N,h)\) be a compact Riemannian manifold with non-positive curvature and with either empty or convex boundary. 
Consider a map \(f\colon M\to N\). 
\begin{enumerate}
    \item For any $u\in C^\infty(\partial M,N)$, the Dirichlet problem 
     \begin{align*}
         \tau(f)=&0 \text{ on M }, \\
         f=&u \text{ on } \partial M
     \end{align*}
     has a solution in every relative homotopy class i.e., every homotopy class of maps $f\colon M\to N$ with $f_{|_{\partial M}}=u$.
    \item The Neumann problem
    \begin{align*}
         \tau(f)=&0 \text{ on M }, \\
         df(\nu)=&0 \text{ on } \partial M
     \end{align*}
     has a solution in every homotopy class.
\end{enumerate}
\end{Satz}

In \cite[Theorems  1.2 \& 1.4]{MR664104} Lemaire established the following
existence result:
\begin{Satz}\label{t:Lemaire}
Let \((M,g)\) be a compact surface with boundary and \((N,h)\) a Riemannian manifold with \(\pi_2(N)=0\). 
\begin{enumerate}
\item If $N$ is compact, then any homotopy class of maps from \(M\) to \(N\)
contains a smooth solution of the Neumann problem
\begin{align*}
 \tau(f)=0,\qquad df(\nu)\big|_{\partial M}=0   
\end{align*}
minimizing \(E(f)\) in the class.
\item If $(N,h)$ is homogeneously regular, then for any smooth map $u\colon \partial M\to N$ and any relative homotopy class $[f]$ of maps $M\to N$, there exists a smooth harmonic map $f_0\colon M\to N$ with $f_0\in[f]$, in particular such that $f_0{}_{|_{\partial M}}=u$.
\end{enumerate}
\end{Satz}

Here, ``homogeneously regular'' means there exists an open covering of $N$ by relatively compact chart domains such that the metric coefficients in those charts are uniformly bounded below and above on $N$ by those of the Euclidean one.
This is a weaker assumption than compactness.

Actually, \cite[Theorem 1.1]{MR1030856} and the remark that follows imply the existence results from Theorem \ref{t:Lemaire} only assuming $N$ to be compact.

\subsection{Elliptic boundary conditions for the Dirac operator}\label{ss:ellipticbc}

We recall a few basics about elliptic boundary conditions for twisted Dirac operators.
We mainly refer to \cite{zbMATH06748684} and mention good former references such as the founding book \cite[Part III]{MR1233386} or more recent works such as \cite[Chap. 4]{MR2509837}, \cite{MR1946443}  or \cite[Sec. 4]{zbMATH05666397}.\\

Let $(M^n,g)$ be any compact Riemannian spin manifold with nonempty boundary $\partial M$ and $E\to M$ be any Riemannian or Hermitian vector bundle endowed with a metric connection $\nabla^E$ over $M$.
In the present article the bundle $E$ will mostly be the pullback bundle $f^*TN\to M$ for a given map $f\colon M\to N$ to some Riemannian manifold $N$.
We denote by $\Sigma M\to M$ the Hermitian spinor bundle of $TM\to M$, by $\nabla^{\Sigma M}$ the covariant derivative induced by the Levi-Civita connection on $TM$ and by "$\cdot$" the Clifford multiplication $T^*M\otimes\Sigma M\to\Sigma M$.
Let $D^E\colon\Gamma(\Sigma M\otimes E)\circlearrowleft$, $\phi\mapsto\sum_{j=1}^n(e_j\cdot\otimes\mathrm{id})\nabla_{e_j}^{\Sigma M\otimes E}\phi$, be the associated twisted Dirac operator.
It is an elliptic and formally selfadjoint first-order linear differential operator.
Denote by $\Sigma\to\partial M$ the Hermitian bundle defined by $\Sigma:=\Sigma\partial M$ when $n$ is odd and $\Sigma:=\Sigma \partial M\oplus\Sigma \partial M$ when $n$ is even.
Recall that $\partial M$ inherits a spin structure from that of $M$ via the unit normal vector field $\nu$ trivializing the normal bundle $T^\perp\partial M\to\partial M$.
The bundle $\Sigma\to\partial M$ actually coincides with the restricted spinor bundle $\Sigma M_{|_{\partial M}}$.
Then a boundary-adapted operator for $D^E$ in the sense of \cite[Lemma 2.2]{zbMATH06748684} is given by 
\[A_0:=\sum_{j=1}^{n-1}(e_j\cdot\nu\cdot\otimes\mathrm{id})\nabla_{e_j}^{\Sigma\otimes E}\colon\Gamma(\Sigma\otimes E)\circlearrowleft.\]
The operator $A_0$ can be identified with the intrinsic twisted Dirac operator $\displaystyle D_{\partial M}^E:=\sum_{j=1}^{n-1}(e_j\cdot_{\partial M}\otimes\mathrm{id})\nabla_{e_j}^{\Sigma\partial M\otimes E}$ on $\partial M$ when $n$ is odd and with $D_{\partial M}^E\oplus-D_{\partial M}^E$ when $n$ is even respectively.
As a consequence, the operator $D_{\partial M}^E$ being elliptic and self-adjoint, the spectrum of $A_0$ is discrete, real and always symmetric about $0$, therefore its $\eta$-invariant $\eta(A_0)$ vanishes.
Mind that we will be later interested in the $\eta$-invariant of $A=D_{\partial M}^E$, not of $A_0$.\\

As in \cite[Sec. 3.6]{zbMATH06748684}, a boundary condition for $D^E$ will be described in this article as a closed subspace $B$ of $\check{H}(A_0):=H_{(-\infty,a)}^{\frac{1}{2}}(A_0)\oplus H_{[a,\infty)}^{-\frac{1}{2}}(A_0)$ for some $a\in\R$, where $H_{(-\infty,a)}^{\frac{1}{2}}(A_0)$ is the closure of $\displaystyle\bigoplus_{\lambda_i<a}\ker(A_0-\lambda_i)$ in $H^{\frac{1}{2}}(\partial M,\Sigma\otimes E)$, the definition of $H_{[a,\infty)}^{-\frac{1}{2}}(A_0)$ being analogous.
Mind that the value of $a$ plays no role since all such $\check{H}(A_0)$ are equivalent to each other for different values of $a$.
The boundary condition $B$ will be called pseudo-local if it is of the form $B=P(H^{\frac{1}{2}}(\partial M,\Sigma\otimes E))\subset L^2(\partial M,\Sigma\otimes E)$, where $P\colon\Gamma(\partial M,\Sigma\otimes E)\circlearrowleft$ is a $0$-order linear pseudodifferential operator.
We look for $D^E$-elliptic pseudo-local boundary conditions, where ellipticity can be characterized by $\sigma_P(\xi)_x\colon(\Sigma\otimes E)_x\circlearrowleft$ restricting to an isomorphism $\displaystyle\bigoplus_{\lambda<0}\ker(i\sigma_{A_0}(\xi)_x-\lambda)\to\sigma_P(\xi)_x(\Sigma\otimes E)_x$ for all $\xi\in T_x^*\partial M\setminus\{0\}$ and $x\in\partial M$, see \cite[Theorem 3.15]{zbMATH06748684}.
Moreover, a boundary condition $B$ for $D^E$ will be called \emph{self-adjoint} if and only if $B=B^{\rm ad}$, where $B^{\rm ad}:=\left\{\varphi\in\check{H}(A_0)\,|\,\left(\sigma_{D^E}(\nu^\flat)\varphi,\psi\right)=0\;\forall\psi\in B\right\}$ is the adjoint boundary condition.
Because of $\sigma_{D^E}(\nu^\flat)=\nu\cdot$, which is the short notation for $\nu\cdot\otimes\mathrm{id}_E$, we can rewrite $B^{\rm ad}=(\nu\cdot B)^\perp\cap\check{H}(A_0)$, where the orthogonality is w.r.t. the natural pairing in $\check{H}(A_0)$.\\

Because of identity \eqref{eq:firstvariationenergy} below, we will often handle elliptic boundary conditions which are symmetric, the latter meaning the twisted Dirac operator restricted to the space of smooth sections satisfying the boundary condition is symmetric.
This is equivalent to $B\subset B^{\rm ad}$, that is, to $\displaystyle \int_{\partial M}\langle\varphi,\nu\cdot\psi\rangle\,d\mu_g=0$ for all twisted spinors $\varphi,\psi$ on $M$ satisfying the boundary condition.
Note that this condition is weaker than self-adjointness.
For instance, the APS boundary condition is such a condition: it is elliptic and the symmetry condition is fulfilled, however the APS boundary condition is self-adjoint if and only if $\ker(A_0)=\{0\}$, see \cite[Example 5.12]{zbMATH06748684} and below.
Note also that only $\displaystyle\Re\left(\int_{\partial M}\langle\varphi,\nu\cdot\psi\rangle\,d\mu_g=0\right)$ for all twisted spinors $\varphi,\psi$ on $M$ satisfying the boundary condition is required for Dirac-harmonic maps.\\

We will consider the following boundary conditions for the Dirac operator:
\begin{enumerate}
    \item generalized Atiyah-Patodi-Singer (gAPS): define, for any $a\in(-\infty,0]$, 
    \[B(a):=H_{(-\infty,a)}^{\frac{1}{2}}(A_0)\subset L^2(\partial M,\Sigma\otimes E),\]
    with $H_{(-\infty,a)}^{\frac{1}{2}}(A_0)$ as above.
    For $a=0$, that condition reduces to the Atiyah-Patodi-Singer (APS) boundary condition and will be denoted by $B_{\rm APS}$.
    Then $B(a)$ defines an elliptic boundary condition \cite[Example 3.21]{zbMATH06748684} for $D^E$, which is self-adjoint if and only if $H_{[a,-a]}^{\frac{1}{2}}(A_0)=\{0\}$ (so, for $a=0$, it just means $\ker(A_0)=\{0\}$).
    Namely, for $a\leq0$, we have $B(a)^{\rm ad}=B(a)\oplus H_{[a,-a]}^{\frac{1}{2}}(A_0)$ and therefore $B(a)\subset B(a)^{\rm ad}$, which means $B(a)$ always makes $D^E$ symmetric; moreover, $B(a)= B(a)^{\rm ad}$ if and only if $H_{[a,-a]}^{\frac{1}{2}}(A_0)=\{0\}$, which is equivalent to $A$ having no eigenvalue between $a$ and $-a$.
    Note that, for $a>0$, the boundary condition $B(a)$ can still be defined, however it does make $D^E$ even symmetric since $(\varphi,\psi)\mapsto\left(\sigma_{D^E}(\nu^\flat)\varphi,\psi\right)$ does not vanish identically on $B(a)$ and hence $B(a)\not\subset B(a)^{\rm ad}$.
    \item chiral \\
     \begin{align*}
         B_{\rm CHI}^\pm:=\ker\left(\operatorname{Id}\mp\nu\cdot G\right),
     \end{align*}
     where $G=\mathcal{G}\otimes\mathrm{id}_E$ comes from a chirality operator $\mathcal{G}\in\Gamma(M,\mathrm{End}(\Sigma M))$ i.e., $\mathcal{G}$ is a parallel involutive Hermitian field of endomorphisms of the spinor bundle $\Sigma M$ which anti-commutes with Clifford multiplication by tangent vectors on $M$.
     Note that $G$ inherits those properties from $\mathcal{G}$ i.e., $G$ is an involutive Hermitian field of endomorphisms of the twisted spinor bundle $\Sigma M\otimes E$ which anti-commutes with $X\cdot_M\otimes\mathrm{id}_E$, for every tangent vector $X$ on $M$.
     Both $G$ and its restriction to $\Sigma\otimes E$ will be denoted the same.
     Then $B_{\rm CHI}^\pm$ defines an elliptic self-adjoint boundary condition for $D^E$, see e.g. \cite[Example 3.20]{zbMATH06748684}.
    \item\label{item:bcmitbag} MIT bag
      \begin{align*}
         B_{\rm MIT}^\pm:=\ker\left(\operatorname{Id}\mp i\nu\cdot\right)
     \end{align*}
     Mind that $B_{\rm MIT}^\pm$, though elliptic, is not self-adjoint; it does not make $D^E$ even symmetric.
     Namely since, for all $\psi,\varphi\in B_{\rm MIT}^\pm$, one has $\Re(\langle\psi,\nu\cdot\varphi\rangle)=\pm\Re(\langle\psi,i\varphi\rangle)=\pm\Im(\langle\varphi,\psi\rangle)$, there is no reason for $\Re\left(\varphi,\nu\cdot\psi\right)_{L^2(\partial M)}$ to vanish and therefore the MIT bag boundary condition must be excluded from our variational ansatz.
    \item modified generalized Atiyah-Patodi-Singer (mgAPS): 
    define, for any $a\in(-\infty,0]$, 
    \[B_{\rm mgAPS}(a):=\left\{\varphi\in\check{H}(A_0),\,\varphi+\nu\cdot\varphi\in H_{(-\infty,a)}^{\frac{1}{2}}(A_0)\right\},\]
    where, as above, $\nu\cdot$ stands for $\nu\cdot\otimes\mathrm{id}_E$.
    Then, as for the generalized APS boundary condition, $B_{\rm mgAPS}(a)$ is elliptic \cite[Example 3.22]{zbMATH06748684} and is self-adjoint if and only if $H_{[a,-a]}^{\frac{1}{2}}(A_0)=\{0\}$.
    Namely, splitting any $\varphi\in\check{H}(A_0)$ in the form $\varphi=\varphi_{(-\infty,a)}+\varphi_{[a,-a]}+\varphi_{(-a,\infty)}$, we have $\varphi+\nu\cdot\varphi\in H_{(-\infty,a)}^{\frac{1}{2}}(A_0)$ if and only if $\varphi_{[a,-a]}=0$ and $\varphi_{(-a,\infty)}=-\nu\cdot\varphi_{(-\infty,a)}$ using again the fact that $\{A_0,\nu\cdot\}=0$.
    In other words, $B_{\rm mgAPS}(a)=\left\{\varphi-\nu\cdot\varphi,\,\varphi\in H_{(-\infty,a)}^{\frac{1}{2}}(A_0)\right\}$.
    Therefore $B_{\rm mgAPS}(a)^{\rm ad}=B_{\rm mgAPS}(a)\oplus H_{[a,-a]}^{\frac{1}{2}}(A_0)$.
\item transmission conditions: given any closed (it could be assumed complete though, and even to have a nonempty boundary) Riemannian spin manifold $(\tilde{M}^n,\tilde{g})$ and any closed embedded orientable hypersurface $N$ in $\tilde{M}$, one may cut $\tilde{M}$ along $N$ and obtain a new compact Riemannian manifold with boundary, namely $\displaystyle M^n:=(\tilde{M}\setminus N)\dot{\cup} N\dot{\cup}(-N$) with induced metric $g$, where $-N$ denotes $N$ with the reverse induced orientation.
This is due to the fact that, when splitting $\partial M$ into two copies of $N$, the inward unit normal along one copy is the outward one along the other copy.
Then $M$ is spin and the twisted Dirac operator $\tilde{D}^E$ induces a twisted Dirac operator $D^E$ acting on sections of $\Sigma M\otimes E\to M$.
Moreover, $A_0\oplus -A_0$ becomes a $D^E$-adapted boundary operator.
Define 
\[B_{\rm trans}:=\left\{(\varphi,\varphi)\,|\,\varphi\in H^{\frac{1}{2}}(N,E)\right\}\subset \check{H}(A_0\oplus -A_0).\]
Then $B_{\rm trans}$ defines an elliptic self-adjoint boundary condition for $D^E$ on $M$, which moreover satisfies $\displaystyle \mathrm{ind}(\tilde{D}^E)=\mathrm{ind}(D^E)_{B_{\rm trans}}$ \cite[Example 3.23 \& Theorem 4.5]{zbMATH06748684}.
\end{enumerate}

\subsection{Index Theorem with boundary}\label{ss:indexthmwithboundary}

When Atiyah-Patodi-Singer (APS) boundary conditions are imposed along the boundary, then the Atiyah-Patodi-Singer index theorem \cite{MR397797,MR397798,MR397799} for a twisted Dirac operator $D^E$ on an even-dimensional compact Riemannian manifold $(M^n,g)$ as above states that
\begin{align}\label{t:APSindexthm}
\mathrm{ind}(D_+^E)=(\hat{A}(TM)\cdot\ch(E))[M]+T(\hat{A}(TM)\cdot\ch(E))[\partial M]-\frac{h(A)+\eta(A)}{2},
\end{align}
where $\mathrm{ind}(D_+^E):=\dim(\ker(D_+^E))-\dim(\mathrm{coker}(D_+^E))$ is the index of $D_+^E$, \(h(A):=\dim(\ker(A))\) and the eta-invariant $\eta(A)$ of $A$ is the value at $0$ of the unique meromorphic extension of the eta function, which is defined by
\begin{align*}
\eta_A(s):=\sum_{\lambda\in\spec A,\lambda\neq 0}(\text{sign }\lambda)|\lambda|^{-s}   
\end{align*}
for every complex number $s$ of sufficiently large real part.
Moreover, \(\hat A(TM)\) represents the \(\hat A\)-genus of the tangent bundle of the manifold \(M\), \(\ch(E)\) the Chern character of the vector bundle \(E\) and \(T\) denotes the so-called transgression form.
Recall that, in even dimension $n$, the complex spinor bundle $\Sigma M\to M$ of $M$ splits under the Clifford action of the complex volume form into $\Sigma_+M\oplus \Sigma_-M$ and $D_+^E\colon\Gamma(\Sigma_+M\otimes E)\to\Gamma(\Sigma_-M\otimes E)$.
\\
Mind that the boundary operator $A$ here is not the boundary-adapted operator $A_0$ that we need to define elliptic boundary conditions for the twisted Dirac operator.
Actually, $A=D_{\partial M}^E$, the boundary twisted Dirac operator, so $A_0=A$ in case $n=\dim(M)$ is odd, however $A_0=A\oplus -A$ in case $n$ is even.\\

In case $n\in2+4\mathbb{Z}$, the Atiyah-Patodi-Singer index formula simplifies, at least when the twist bundle $E\to M$ is real.
Namely, since $\hat{A}(TM)$ is a polynomial in the Pontryagin classes of the real vector bundle $TM\to M$, its cohomology class is represented by closed forms of degree $4k$ only.
Since $n\in2+4\mathbb{Z}$, only the odd Chern classes involved in $\mathrm{ch}(E)$ contribute to $\hat{A}(TM)\cdot\mathrm{ch}(E)[M]$.
But, because of $E\to M$ being assumed real, all its odd Chern classes must vanish, see e.g. \cite[Lemma 14.9]{zbMATH03468033}.
Therefore $\hat{A}(TM)\cdot\mathrm{ch}(E)[M]=0$ and $T(\hat{A}(TM)\cdot\mathrm{ch}(E))[\partial M]=0$ for the same reason.
Moreover, because the twisted Dirac operator of the boundary has symmetric spectrum in dimension $1+4\mathbb{Z}$ due to the presence of a real or quaternionic structure on the spinor bundle anti-commuting with the Clifford multiplication by vectors \cite[Section 2.5]{MR2230574}, one has $\eta(A)=0$, so that \eqref{t:APSindexthm} reduces to 
\[\mathrm{ind}(D_+^f)=-\frac{1}{2}h(A),\]
which already implies that $h(A)$ must be even.
Note that there is no need for assuming the boundary to be totally geodesic here.\\

At this point it is important to underline that, when Atiyah-Patodi-Singer boundary conditions for the spinor component are considered, the index of $D_+^f$ is not homotopy-invariant because of the dimension of the kernel of $A=D_{\partial M}^E$ already not being homotopy-invariant.
Namely, on every closed spin manifold of dimension $3+4k$ 
endowed with a fixed twist-bundle with connection $\nabla^E$,  there always exists a Riemannian metric for which the Dirac operator twisted by $E$ has nontrivial kernel \cite[Theorem 4.1]{zbMATH01121705}, nonetheless that kernel is generically trivial \cite{zbMATH05480735}, therefore its dimension depends on the choice of metric and is thus not homotopy-invariant.
Therefore we cannot hope for the dimension of the kernel of $A$ to be nonzero at least for small deformations of the metric or the twist-bundle.
That lack of homotopy invariance of the index under APS boundary conditions had already long been noticed, although the homotopy invariance under local boundary conditions does hold, see e.g. \cite[Remark 22.25]{MR1233386}.
We note also here that the generic minimality of the kernel of $D^f$ in terms of $(g,h,f)$ on a closed manifold $\overline{M}$ was established under certain assumptions in \cite[Theorem 3.4]{zbMATH07094314} but when $\dim(\overline{M})=2$, which is not what we look at here since $\partial M$ is odd-dimensional in most of our applications.\\

Another point of attention is the following.
Assuming that, given a smooth one-parameter-family of smooth maps $(f_t)_{t\in(-\varepsilon,\varepsilon)}$ from $M$ to $N$, there exists, for any $t\in(-\varepsilon,\varepsilon)$, a nonzero $\phi_t\in\ker(D^{f_t})$ satisfying the APS boundary condition, it is {\sl a priori} not always possible to have a \emph{smooth} one-parameter-family of $(\phi_t)_{t\in(-\varepsilon,\varepsilon)}$ of such nonzero spinors.
Even if the harmonic map $f_0$ is assumed to be \emph{perturbation-minimal} in the sense of \cite[Def. 8.1]{MR3070562} that is, that $\dim(\ker(D^{f_t}))\geq\dim(\ker(D^{f_0}))$ holds for every smooth variation of $f_0$, then $\dim(\ker(D^{f_t}))=\dim(\ker(D^{f_0}))$ must actually hold for all sufficiently small $t$.
Nonetheless the very first step in constructing a smooth one-parameter-family of twisted harmonic spinors $(\phi_t)_{t\in(-\varepsilon,\varepsilon)}$, which would be the statement of \cite[Prop. 8.2]{MR3070562} in our context, fails: although the orthogonal projection $\pi_t$ onto the kernel of $D^{f_t}$ -- and hence onto that of the parallely transported operator $\hat{D}_t$ as in the proof of \cite[Prop. 8.2]{MR3070562} -- can still be defined, the lack of a discrete spectrum of $D^{f_t}$, an operator which is not self-adjoint because of $\ker(A_t)\neq\{0\}$ in most of our applications, does not ensure $\pi_t(\phi_0)$ to converge towards $\pi_0(\phi_0)$ in any suitable topology as $t\to0$, whatever $\phi_0\in\ker(D^{f_0})$ is.

\begin{Bem}\label{r:maphomotopyinvarianceindim2mod4}
In case $n\in 2+4\mathbb{Z}$ and Dirichlet and APS boundary conditions being fixed for some smooth map $f_0\colon M^n\to N^m$ and for the twisted spinor field $\phi_0$ respectively, the index of $D_+^f$ does not depend on $f\in[f_0]$.
Namely, for any one-parameter-family of smooth maps $(f_t)_{t\in(-\varepsilon,\varepsilon)}$ from $M$ to $N$ satisfying $f_t{}_{|_{\partial M}}=f_0{}_{|_{\partial M}}$ for all $t\in(-\varepsilon,\varepsilon)$, one has $\mathrm{ind}(D_+^{f_t})=-\frac{1}{2}h(A(f_t))=-\frac{1}{2}h(A(f_0))$ for all $t\in(-\varepsilon,\varepsilon)$ because of $A(f_t)=A(f_0)$ (the operator $A$ only depends on the values of $f$ along $\partial M$).
This has the interesting consequence that, if $h(A(f_0))\neq0$, then $\dim(\ker(D^{f}))\geq|\mathrm{ind}(D_+^{f})|=\frac{1}{2}h(A(f_0))\neq0$ for any $f$ lying in the same relative homotopy class as $f_0$.
Still, as we noticed above, the question arises to find a smooth one-parameter-family of nonzero harmonic twisted spinors -- all satisfying the same APS boundary condition $B_{\rm APS}$ (depending on $A_0=A_0(f_0)$ only) because of the Dirichlet boundary condition for $f_0$ -- given a smooth one-parameter-family of maps from $M$ to $N$.
Besides, because the index is not homotopy-invariant, we cannot handle the case where $f_0$ is not perturbation-minimal: the integers $b_m$ et $d_m$ from \cite[Sect. 9]{MR3070562} cannot be defined, and if so, then probably vanish.
\end{Bem}

For explicit calculations using the index theorem on manifolds with boundary, in particular with application to physics in mind, we recommend the survey \cite{MR598586}.

\section{Variational Formulas on manifolds with boundary and applications}\label{s:variationalformulas}

\subsection{Variational formula}\label{ss:varformula}
From here on let $(M^n,g)$ be any compact 
Riemannian spin manifold with nonempty smooth boundary $\partial M$.
Let $(N,h)$ be any Riemannian manifold.
As above, we consider pairs $(f,\phi)$ where $f\colon M\to N$ is any smooth map and $\phi\in\Gamma(\Sigma M\otimes f^*TN)$ is any smooth section of the spinor bundle twisted by the pull-back tangent bundle.
Given such a $\phi$, we define the section $V_\phi\in\Gamma(f^*TN)$ by duality via
\[h(V_\phi,Y):=\sum_{j=1}^n\langle (e_j\cdot\otimes R_{df(e_j),Y}^N)\phi,\phi\rangle\]
for all tangent vectors $Y$ to $N$, where $\{e_j\}_{1\leq j\leq n}$ is any local orthonormal basis of $TM$.
Recall that, the expression $\langle (X\cdot\otimes R_{df(X),Y}^N)\phi,\phi\rangle$ being real for all $X,Y$, the vector field $V_\phi$ is real as well.
Then the first variation of the functional $E$ defined by 
\begin{equation}\label{eq:defenergyfunctionalDhmaps}
E(f,\phi):=\frac{1}{2}\int_M|df|^2+\Re(\langle D^f\phi,\phi\rangle)\,d\mu_g
\end{equation}
can be expressed as follows.
Given any smooth one-parameter-family $(f_t)_{t\in(-\varepsilon,\varepsilon)}$ of maps from $M$ to $N$ and any "smooth" one-parameter-family $(\phi_t)_{t\in(-\varepsilon,\varepsilon)}$, where $\phi_t\in\Gamma(\Sigma M\otimes f_t^*TN)$ for all $t$, we have 
\begin{eqnarray}\label{eq:firstvariationenergy}
\frac{d}{dt}(E(f_t,\phi_t))_{|_{t=0}}&=&\int_M\Big(\Re\langle D^{f_0}\phi_0,\partial_t\phi(0)\rangle-h(\mathrm{tr}_g(\nabla df_0)-\frac{1}{2}V_{\phi_0},\partial_tf(0))\Big)\,d\mu_g\\
\nonumber&&+\int_{\partial M}\Big(h(\partial_t f(0),df_0(\nu))+\frac{1}{2}\Re\left(\langle\partial_t\phi(0),\nu\cdot\phi_0\rangle\right)\Big)\,d\mu_g.
\end{eqnarray}
Here, $\nu$ denotes the inner unit normal along $\partial M$ and $\partial_t\phi(0)$ is a short notation for the natural covariant derivative of $(\phi(t,x))_{(t,x)\in(-\varepsilon,\varepsilon)\times M}$ along $\partial_t$ at $t=0$, see e.g. \cite[Prop. 5.1]{MR3070562}.\\

We shall always look for critical points of $E$ where the mapping component $f_0$ satisfies either the Dirichlet or the Neumann boundary condition and the spinor component $\phi_0$ satisfies some boundary condition $\phi_0{}_{|_{\partial M}}\in B$.
We only look for variations $(f_t)_t$ of $f_0$ and $(\phi_t)_t$ of $\phi_0$ satisfying the same boundary condition as $f_0$ and $\phi_0$ respectively.
On the one hand, in case of the Dirichlet boundary condition for $f_0$, this means the whole family $(f_t)_{t\in(-\varepsilon,\varepsilon)}$ of maps from $M$ to $N$ is required to satisfy that same condition i.e., $f_t{}_{|_{\partial M}}=u$ for all $t\in(-\varepsilon,\varepsilon)$ and some fixed map $u\in C^\infty(M,N)$.
As a consequence, whether the boundary condition for $f_0$ being the Dirichlet or the Neumann one, the boundary term $\displaystyle\int_{\partial M}h(\partial_t f(0),df_0(\nu))\,d\mu_g$ will always vanish.
On the other hand, this also requires $\phi_t{}_{|_{\partial M}}\in B$ for all $t\in(-\varepsilon,\varepsilon)$, so that the boundary space $B$ will be chosen so as to satisfy $\int_{\partial M}\Re\left(\langle\psi,\nu\cdot\varphi\rangle\right)\,d\mu_g=0$ for all $\psi,\varphi\in B$.
Here, it is important to note that the term $\partial_t\phi(0)$ is actually the covariant derivative of $s\mapsto\phi_s$ along $\partial_t$ at $s=0$, and there is {\sl a priori} no reason why that covariant derivative should preserve the boundary space $B$ when restricted to $\partial M$.
For instance, there is no reason why the gAPS and the mgAPS boundary conditions satisfy that condition, whatever the boundary condition for the map is.
Assuming $\int_{\partial M}\Re\left(\langle\psi,\nu\cdot\varphi\rangle\right)\,d\mu_g=0$ for all $\psi,\varphi\in B$ holds as well as $\partial_t\phi(0)\in B$ for every variation $(\phi_t)_t$ of some $\phi_0$ satisfying $\phi_t{}_{|_{\partial M}}\in B$ for all $t$, the boundary term $\displaystyle\int_{\partial M}\Re\left(\langle\partial_t\phi(0),\nu\cdot\phi_0\rangle\right)\,d\mu_g$ will automatically vanish.
For example, the chirality boundary condition satisfies that requirement, whether the Dirichlet or the Neumann boundary condition for the map is chosen, since it only involves the $\Sigma M$ part.
If the Dirichlet boundary condition for the map is fixed, then $B$ is automatically preserved by the above covariant derivative -- whatever $B$ is -- since the twist bundle as well as the connection along the boundary do not depend on $t$.\\

Thus, assuming $\int_{\partial M}\Re\left(\langle\psi,\nu\cdot\varphi\rangle\right)\,d\mu_g=0$ for all $\psi,\varphi\in B$ holds as well as $\partial_t\phi(0)\in B$ for every variation $(\phi_t)_t$ of some $\phi_0$ satisfying $\phi_t{}_{|_{\partial M}}\in B$ for all $t$, a Dirac-harmonic map in our context will be a critical point of $E$ 
that is, a pair $(f_0,\phi_0)$ satisfying 
\begin{equation}\label{eq:DhmapsonMwithboundary}
\left\{\begin{array}{ll}D^{f_0}\phi_0&=0\,,  \\ 
\mathrm{tr}_g(\nabla df_0)&=\frac{1}{2}V_{\phi_0}\end{array}\right.
\end{equation}
as well as $f_0{}_{|_{\partial M}}=u$ or $df_0(\nu)=0$ on $\partial M$ respectively and $\phi_0{}_{|_{\partial M}}\in B$.\\

In case the above assumptions on $B$ are not satisfied, we can just declare Dirac-harmonic maps to be pairs $(f_0,\phi_0)$ satisfying \eqref{eq:DhmapsonMwithboundary} as well as the corresponding required boundary conditions.
But then we must stay aware that those Dirac-harmonic maps do not fit with our variational ansatz, i.e.
they are not critical points of the energy functional \eqref{eq:defenergyfunctionalDhmaps}.\\

In that setting the main existence result is the following, where a Dirac-harmonic map $(f_0,\phi_0)$ is called \emph{uncoupled} if $V_{\phi_0}=0$ (see \cite[Definition 1.1]{MR3070562}):

\begin{Prop}\label{p:existuncoupledDhfromvarformula}
Assume $f_0\colon M\to N$ to be a harmonic map with either Dirichlet or Neumann boundary condition
and $\phi_0\in\ker(D^{f_0})$ such that $\phi_0{}_{|_{\partial M}}\in B$, where $B$ is chosen so that $\displaystyle\int_{\partial M}\Re\left(\langle\psi,\nu\cdot\varphi\rangle\right)\,d\mu_g=0$ for all $\psi,\varphi\in B$.
Assume also that, for every smooth variation $(f_t)_{t\in(-\varepsilon,\varepsilon)}$ with either Dirichlet or Neumann boundary condition of $f_0$, there exists a variation $(\phi_t)_{t\in(-\varepsilon,\varepsilon)}$ of $\phi_0$ such that $\displaystyle\frac{d}{dt}\Re\left(D^{f_t}\phi_t,\phi_t\right)_{L^2(M)}{}_{|_{t=0}}=0$ and $\phi_t{}_{|_{\partial M}}\in B$ for all $t\in(-\varepsilon,\varepsilon)$.
Then $(f_0,\phi_0)$ is an uncoupled Dirac-harmonic map with Dirichlet or Neumann boundary condition for $f_0$ and boundary condition $B$ for $\phi_0$.
\end{Prop}

\begin{proof}
The proof goes {\sl verbatim} as that of \cite[Cor. 5.2]{MR3070562}: since, by assumption, $\displaystyle\frac{d}{dt} E(f_t,\phi_t)_{|_{t=0}}=\frac{1}{2}\frac{d}{dt}\int_M|df_t|^2\,d\mu_g{}_{|_{t=0}}=0$ holds, and the boundary terms vanish, the pair $(f_0,\phi_0)$ must be Dirac-harmonic.
\end{proof}

Again, the vanishing of the boundary term $\displaystyle\Re\left(\partial_t\phi(0),\nu\cdot\phi_0\right)_{L^2(\partial M)}$ holds as soon as (but not only when) $B$ is a symmetric boundary condition for $D^{f_0}$ and $\partial_t\phi(0)\in B$ for every variation $(\phi_t)_t$ of some $\phi_0$ satisfying $\phi_t{}_{|_{\partial M}}\in B$ for all $t$.
This will be always fulfilled in case $f_0$ satisfies the Dirichlet boundary condition for $f_0$, however not for every $B$ in case $f_0$ satisfies the Neumann boundary condition.\\

As in the closed case, there exist obvious examples of uncoupled Dirac-harmonic maps.
First, take any compact Riemannian spin manifold $(M^n,g)$ with boundary admitting a nonzero harmonic spinor $\psi_0$ with given arbitrary boundary condition $\psi_0{}_{|_{\partial M}}\in B$ and any \emph{constant} map $f_0\colon M\to N^m$ to any Riemannian manifold $(N^m,h)$, then the pair $(f_0,\phi:=\psi_0^{\oplus^m})$ is a Dirac-harmonic map with Neumann boundary condition for $f_0$ as well as $\phi_0{}_{|_{\partial M}}\in B^{\oplus^m}$.
Another category of trivial examples consists of those pairs $(f_0,0)$, where $f_0$ is a harmonic map with the required boundary condition.\\

In the sequel, we want to avoid those trivial examples.

\subsection{Application of the variational formula}\label{ss:appvarformula}

As in the closed case, if a $\mathbb{Z}_2$-grading on $\Sigma M$ is available, uncoupled Dirac-harmonic maps can be obtained:

\begin{Prop}\label{p:examplesDhmapsZ2grading}
Assume the existence of a parallel involutive real or complex-linear, symmetric or Hermitian, endomorphism field $\mathcal{G}$ 
of $\Sigma M$ which anti-commutes with the Clifford multiplication by tangent vectors on $M$.
Split $\Sigma M=\Sigma_+M\oplus\Sigma_-M$, where $\Sigma_\pm M:=\ker(\mathcal{G}\mp\mathrm{id})\subset\Sigma M$.
Let $f_0\colon M\to N$ be any harmonic map satisfying the Dirichlet or Neumann boundary condition.
Let $B\subset \check{H}(A_0)$ be any elliptic boundary condition for $D^{f_0}$ satisfying $\partial_t\phi(0)\in B$ for every variation $(\phi_t)_t$ of some $\phi_0$ satisfying $\phi_t{}_{|_{\partial M}}\in B$ for all $t$.
Assume also that $\mathcal{G}\otimes\mathrm{id}_{|_{\partial M}}$, seen as a pointwise endomorphism field of $\Sigma M\otimes f_0^*TN$, preserves $B$, i.e., $(\mathcal{G}\otimes\mathrm{id})(B)\subset B$.
Then for any $\phi_0\in \ker(D^{f_0})$ with $\phi_0{}_{|_{\partial M}}\in B$, both $(f_0,\phi_0{}_\pm)$ are uncoupled Dirac-harmonic maps.
\end{Prop}

\begin{proof}
The proof follows exactly that of \cite[Cor. 6.1]{MR3070562}: given any variation $(f_t)_{t\in(-\varepsilon,\varepsilon)}$ satisfying the Dirichlet or Neumann boundary condition, consider an arbitrary variation $(\phi_t)_{t\in(-\varepsilon,\varepsilon)}$ of $\phi_0$ with $\phi_t{}\in\Gamma(\Sigma M\otimes f_t^*TN)$ and $\phi_t{}_{|_{\partial M}}\in B$ for all $t\in(-\varepsilon,\varepsilon)$.
As above let $\phi_t=\phi_t{}_++\phi_t{}_-$ with $\phi_t{}_\pm\in\Gamma(\Sigma_\pm M\otimes f_t^*TN)$, for all $t\in(-\varepsilon,\varepsilon)$.
Since $D^{f_0}\phi_0=0$ and $\phi_0{}_{|_{\partial M}}\in B$, we also have $D^{f_0}\phi_0{}_\pm=0$ as well as $\phi_0{}_\pm{}_{|_{\partial M}}\in B$, using $\{D^{f_0}, \mathcal{G}\otimes\mathrm{id}\}=D^{f_0}\circ(\mathcal{G}\otimes\mathrm{id})+(\mathcal{G}\otimes\mathrm{id})\circ D^{f_0}=0$ 
and $(\mathcal{G}\otimes\mathrm{id})(B)\subset B$. 
Moreover, because of $\mathcal{G}$ anti-commuting with the Clifford multiplication by tangent vectors on $M$, one has $(X\cdot\otimes R_{Y,Z}^N)\phi_0^\pm\in\Gamma(\Sigma_\mp M\otimes f_0^*TN)$ for all $X\in\Gamma(TM)$ and $Y,Z\in\Gamma(f_0^*TN)$, so that $V_{\phi_0{}_\pm}=0$ by the property of $G$ being symmetric or Hermitian. 
Therefore, only $\Re\left(\partial_t\phi_\pm(0),\nu\cdot\phi_0{}_\pm\right)_{L^2(\partial M)}=0$ has to be shown.
But that term vanishes since $\partial_t\phi_\pm(0)\in\Gamma(\Sigma_\pm M_{|_{\partial M}}\otimes f_0^*TN)$ and $\nu\cdot\phi_0{}_\pm\in\Gamma(\Sigma_\mp M_{|_{\partial M}}\otimes f_0^*TN)$, so that $\Re(\langle\partial_t\phi_\pm(0),\nu\cdot\phi_0{}_\pm\rangle)=0$ and, replacing $\phi_t$ by $\phi_t{}_\pm$ in \eqref{eq:firstvariationenergy}, we obtain $\displaystyle\frac{d}{dt}(E(f_t,\phi_t{}_\pm))_{|_{t=0}}=0$ and the claim.
\end{proof}

\begin{Bem}
Actually, Proposition \ref{p:existuncoupledDhfromvarformula} could be applied to prove Proposition \ref{p:examplesDhmapsZ2grading}, provided $B_\pm:=B\cap\left(H^{\frac{1}{2}}(\partial M,\Sigma_\pm M\otimes f_0^*TN)\oplus H^{-\frac{1}{2}}(\partial M,\Sigma_\pm M\otimes f_0^*TN)\right)$ is shown to be an elliptic boundary condition for $D^{f_0}$.
This should be the case, even if it is not completely clear from \cite[Def. 3.7]{zbMATH06748684}.
\end{Bem}

For instance, a $\mathbb{Z}_2$-grading is provided by a chirality operator as above.
In the particular case where $n$ is even and the $\mathbb{Z}_2$-grading $\mathcal{G}$ is provided by the pointwise Clifford action of the complex volume form $\omega_n^{\mathbb{C}}=i^{\lfloor\frac{n+1}{2}\rfloor}e_1\cdot\ldots\cdot e_n$ on $\Sigma M$, the condition $(\mathcal{G}\otimes\mathrm{id})(B_{\mathrm{gAPS}}(a))\subset B_{\mathrm{gAPS}}(a)$ is satisfied for any $a\leq0$:
\begin{Lem}\label{l:GcommuteswithA0}
Let $(M^n,g)$ be any even-dimensional compact Riemannian spin manifold with nonempty boundary and $E\to M$ by any Riemannian real vector bundle with metric connection over $M$.
Let $\mathcal{G}:=\omega_n^{\mathbb{C}}\cdot\colon\Gamma(\Sigma M)\circlearrowleft$ be the $\mathbb{Z}_2$-grading provided by the Clifford action of the complex volume form of $M$.
Then $(\mathcal{G}\otimes\mathrm{id})(B_{\mathrm{gAPS}}(a))\subset B_{\mathrm{gAPS}}(a)$ holds for any $a\leq0$.
\end{Lem}
\begin{proof}
The complex volume form $\omega_n^{\mathbb{C}}$ acts pointwise as $\mathrm{Id}_{\Sigma_+ M\otimes E}\oplus-\mathrm{Id}_{\Sigma_-M\otimes E}$ -- and the same on $\Gamma(\Sigma\otimes E)$ -- and the boundary operator $A_0=A\oplus -A$ acts diagonally on $\Gamma(\Sigma\otimes E)$, so that the Clifford action of the complex volume form of $M$ commutes with $A_0$ and therefore $(\mathcal{G}\otimes\mathrm{id})(B_{\rm gAPS}(a))\subset B_{\rm gAPS}(a)$ is fulfilled.
\end{proof}

A fundamental class of examples is obtained by shifting the spectral bound $a$ in $B_{\rm gAPS}(a)$:
\begin{Thm}\label{t:existDhmapsgAPSasmall}
Let $(M^n,g)$ be any even-dimensional compact Riemannian spin manifold with nonempty boundary and $(N^m,h)$ be any compact Riemannian manifold of nonpositive curvature and with either empty or convex boundary. 
Let $[f]$ be any relative homotopy class of maps $M\to N$ satisfying the Dirichlet boundary condition $f_{|_{\partial M}}=u$ for some given $u\in C^\infty(\partial M,N)$.
Then there is a real $a_0\leq0$ such that, for any real $a\leq a_0$, there exists a nontrivial uncoupled Dirac-harmonic map $(f_0,\phi_0)$ with $f_0\in [f]$ and $\phi_0{}_{|_{\partial M}}\in B_{\rm gAPS}(a)$.
\end{Thm}
\begin{proof}
Theorem \ref{t:Hamilton} already ensures the existence of a harmonic map $f_0\in [f]$, in particular $f_0{}_{|_{\partial M}}=u$.
Considering the generalized Atiyah-Patodi-Singer boundary condition $B=B_{\rm gAPS}(a)$ for some $a\in(-\infty,0)$, one has \cite[Example 4.3]{zbMATH06748684}
\[\mathrm{ind}((D_+^{f_0})_{B_{\rm gAPS}(0)})=\mathrm{ind}((D_+^{f_0})_{B_{\rm gAPS}(a)})+\dim(L_{[a,0)}^2(A_0))\]
in arbitrary even dimension $n$.
But since $\dim(L_{[a,0)}^2(A_0))\underset{a\to-\infty}{\longrightarrow}\infty$, there exists $a_0\in(-\infty,0]$ such that $\mathrm{ind}((D_+^{f_0})_{B_{\rm gAPS}(a)})<0$ for all $a\leq a_0$.
Then, for every $a\leq a_0$, the existence of a $\Phi_0\in \ker(D^{f_0})\setminus\{0\}$ satisfying $\Phi_0{}_{|_{\partial M}}\in B_{\rm gAPS}(a)$ follows.
Since, for the natural $\mathbb{Z}_2$-grading $\mathcal{G}$ given by the pointwise Clifford action of the complex volume form, $(\mathcal{G}\otimes\mathrm{id})(B_{\mathrm{gAPS}}(a))\subset B_{\mathrm{gAPS}}(a)$ is satisfied by Lemma \ref{l:GcommuteswithA0}, Proposition \ref{p:examplesDhmapsZ2grading} can be applied.
It can be deduced that both $(f_0,\Phi_0{}_\pm)$ are Dirac-harmonic maps satisfying $\Phi_0{}_\pm{}_{|_{\partial M}}\in B_{\mathrm{gAPS}}(a)$ and, because $\Phi_0{}_+\neq0$ or $\Phi_0{}_-\neq0$, one can take $\phi_0:=\Phi_0{}_+$ or $\Phi_0{}_-$ according to which one does not vanish identically.
This concludes the proof.
\end{proof}

\begin{Bem}\label{r:shiftotherBC}
Whether other boundary conditions for the spinor field can be handled in the same way using \cite[Theorem 4.4]{zbMATH06748684}, where the shift from the gAPS to an arbitrary elliptic boundary condition is used, remains presently open.
\end{Bem}

In dimension $2$ mod $4$, because of $\mathrm{ind}(D_+^f)=-\frac{1}{2}h(A)$ as shown above, we have the following general statement:
\begin{Thm}\label{t:n=2+4kAPSDirichlet}
Let $M^n$ be any compact spin manifold with nonempty boundary, dimension $n=2+4k\geq6$ and such that, for any any real twist vector bundle $E\lra \partial M$, there exists a Riemannian metric on $\partial M$ for which $\ker(D_{\partial M}^E)\neq\{0\}$ holds.
Let $(N^m,h)$ be any compact Riemannian manifold of nonpositive curvature and with either empty or convex boundary. 
Let $[f]$ be any relative homotopy class of maps $M\to N$ satisfying the Dirichlet boundary condition $f_{|_{\partial M}}=u$ for some given $u\in C^\infty(\partial M,N)$.
Then there exists a smooth Riemannian metric $g$ on $M$ for which a nontrivial uncoupled Dirac-harmonic map $(f_0,\phi_0)$ exists with $f_0\in[f]$ and $\phi_0{}_{|_{\partial M}}\in B_{\rm APS}$.
\end{Thm}
\begin{proof}
By assumption, there exists a smooth Riemannian metric on $\partial M$ such that $\ker(D_{\partial M}^u)\neq\{0\}$.
Extend that metric to an arbitrary smooth Riemannian metric $g$ on $M$.
Because, as noticed above, the condition $(\mathcal{G}\otimes\mathrm{id})(B_{\mathrm{APS}})\subset B_{\mathrm{APS}}$ still holds, Proposition \ref{p:examplesDhmapsZ2grading} can be applied in the same way, which concludes the proof.
\end{proof}

Up to our knowledge, the existence of a Riemannian metric with non-vanishing kernel of the corresponding Dirac operator twisted with any fixed vector bundle has not been established in dimension $1+4k\geq5$, the only result for twisted Dirac operators being \cite[Theorem 4.1]{zbMATH01121705}, which is valid in dimension $3+4k$.
If that existence holds, then Theorem \ref{t:n=2+4kAPSDirichlet} is valid in dimension $2+4k\geq6$ without the assumption on $\partial M$.
For the $2$-dimensional case, we refer to Section \ref{ss:dimM=2} below.

\begin{Bem}\label{r:Ammannuncoupled}
In \cite[Theorem 1]{bernd}, Ammann observed that in the case of a compact domain,
Dirac-harmonic maps are most probably always uncoupled and pointed out that the same might be true in the case of a compact manifold with boundary.
\end{Bem}

\section{Explicit examples via twistor spinors}\label{s:firstexamplesviatwistorspinors}
In this section we derive a number of explicit solutions for the boundary value problem of Dirac-harmonic maps based on twistor spinors. As only few Riemannian manifolds admit twistor spinors we expect that this strategy can only produce a small number of explicit examples.

\subsection{Via twistor spinors}\label{ss:exampleswithtwistorspinors}
Given a smooth map $f\colon M\to N$, we make the ansatz 
\begin{align}
\label{dfn:psi-twistor}
\phi:=\sum_{j=1}^n e_j\cdot\psi\otimes f_*(e_j)+\varphi\otimes\tilde{\nu},
\end{align}
where $\psi,\varphi\in\Gamma(\Sigma M)$ are (untwisted) spinor fields, $\{e_j\}_{1\leq j\leq n}$ is any local o.n.b. of $TM$ and $\tilde{\nu}\in\Gamma(M,f^*TN)$ is any vector field standing orthogonally to $f_*(TM)=df(TM)$ everywhere on $M$.
It is well-known -- see e.g. \cite[Lemma 2.1]{zbMATH07075240} -- that
\begin{eqnarray*}
D^f\phi&=&\sum_{j=1}^n\left(\frac{2-n}{n}e_j\cdot D_M\psi-2P_{e_j}\psi\right)\otimes f_*(e_j)-\psi\otimes\mathrm{tr}_g(\nabla df)\\
& &{}+(D_M\varphi)\otimes\tilde{\nu}+\sum_{j=1}^ne_j\cdot\varphi\otimes\nabla_{e_j}^N\tilde{\nu},
\end{eqnarray*}
where $D_M\colon\Gamma(\Sigma M)\circlearrowleft$ is the standard spin Dirac operator of $(M^n,g)$.
It is also known -- see e.g. \cite[Cor. 2.3]{zbMATH07075240} -- that, when $n=2$, the map $f$ is harmonic, $\varphi=0$ and $\psi$ is a twistor-spinor on $(M,g)$, the pair $(f,\phi)$ is an uncoupled Dirac-harmonic map, where no boundary condition is required.
Thus the question is whether both $f$ and $\psi$ can be chosen such that $(f,\phi)$ satisfy  boundary conditions as described above.\\

We only consider the Neumann boundary condition for $f$, which comes from the following elementary question.
Namely when does $\phi$ inherit the boundary condition from $\psi$?
For instance, consider $\psi$ to satisfy the chirality boundary condition $\nu\cdot\mathcal{G}\psi=\psi$ along $\partial M$, where $\nu$ here denotes the inner unit normal along $\partial M$.
Taking $G:=\mathcal{G}\otimes\mathrm{id}_E$ and $\phi:=\sum_{j=1}^ne_j\cdot\psi\otimes f_*(e_j)$ as above with $\varphi=0$, we have, in a pointwise o.n.b. $\{e_j\}_{1\leq j\leq n}$ of $T_xM$, $x\in\partial M$, with $e_1=\nu_x$:
\begin{eqnarray*}
\nu\cdot G\phi&=&-\nu\cdot\nu\cdot\mathcal{G}\psi\otimes f_*(\nu)+\sum_{j=2}^ne_j\cdot\nu\cdot\mathcal{G}\psi\otimes f_*(e_j)\\
&=&-\nu\cdot\psi\otimes f_*(\nu)+\sum_{j=2}^ne_j\cdot\psi\otimes f_*(e_j),
\end{eqnarray*}
 $\nu\cdot G\phi=\phi$ if and only if $\nu\cdot\psi\otimes f_*(\nu)=0$, the identity
 $\nu\cdot G\phi=-\phi$ being even harder to realize.
When $\psi$ is assumed to vanish nowhere, the condition $\nu\cdot\psi\otimes f_*(\nu)=0$ is equivalent to $f_*(\nu)=0$,  hence we naturally first focus on the Neumann boundary condition for $f$.
The same argument shows that $i\nu\cdot\phi=-\phi$ as soon as $i\nu\cdot\psi=\psi$ and $f_*(\nu)=0$ along $\partial M$.\\

We first handle the chirality boundary condition for the spinor component.
Assume $\nu\cdot\mathcal{G}\psi=\psi$ to hold along $\partial M$, where $\psi$ is a twistor-spinor on $(M,g)$.
Then, for every $X$ tangent to $\partial M$, denoting by $W$ the Weingarten-endomorphism-field of $T\partial M$,
\begin{eqnarray*}
(\mathrm{id}-\nu\cdot\mathcal{G})\nabla_X\psi&=&\nabla_X\psi-\nu\cdot\nabla_X(\mathcal{G}\psi)\\
&=&\nabla_X\psi-\nabla_X(\nu\cdot\mathcal{G}\psi)+\nabla_X\nu\cdot\mathcal{G}\psi\\
&=&-WX\cdot\mathcal{G}\psi.
\end{eqnarray*}
Since $\psi$ is a twistor-spinor on $(M,g)$, we obtain
\begin{equation}\label{eq:twistorspinorchirality}
X\cdot(\mathrm{id}-\nu\cdot\mathcal{G})D\psi=nWX\cdot\mathcal{G}\psi
\end{equation}
for all $X$ tangent to $\partial M$.
In the particular case where $\psi$ is a nonzero $\alpha$-Killing spinor on $(M,g)$ for some $\alpha\in\R\cup i\R$, we have $D\psi=-n\alpha\psi$, 
\[X\cdot(\mathrm{id}-\nu\cdot\mathcal{G})D\psi=-n\alpha X\cdot(\mathrm{id}-\nu\cdot\mathcal{G})\psi=0\]
and therefore $W=0$ along $\partial M$ by \eqref{eq:twistorspinorchirality} i.e., the boundary $\partial M$ is totally geodesic in $M$.
Assuming now $\partial M$ to be totally geodesic in $M$ and to be connected, the chirality boundary condition $\nu\cdot\mathcal{G}\psi=\psi$ is satisfied as soon as it is at a point in $\partial M$.
This is due to the fact that $\nu\cdot\mathcal{G}$ is then parallel w.r.t. the connection $X\mapsto \nabla_X-\alpha X\cdot$ along $\partial M$ by the above computation.
To sum up, we obtain the following:

\begin{Prop}\label{p:explicitexdim2chiral}
Let $M^2$ be a compact oriented Riemannian surface with connected totally geodesic boundary and fixed spin structure.
Assume $f\colon M^2\to N$ to be a harmonic map satisfying the Neumann boundary condition and $\psi\in\Gamma(M,\Sigma M)$ to be an $\alpha$-Killing spinor for $\alpha\in\R\cup i\R$ on $M$ satisfying $\nu\cdot\mathcal{G}\psi=\psi$ at a point in $\partial M$.
Then $\displaystyle(f,\phi:=\sum_{j=1}^2e_j\cdot\psi\otimes f_*(e_j))$ is a Dirac-harmonic map with Neumann boundary condition $f_*(\nu)=0$ for $f$ and chirality boundary condition $\nu\cdot G\phi=\phi$ for $\phi$.
\end{Prop}

Examples here include the following:
\begin{enumerate}
\item When $M^2=\mathbb{S}^2_+:=\{x=(x_1,x_2,x_3)\in\mathbb{S}^2\,|\,x_3\geq0\}$ with the standard round metric and canonical spin structure, whose boundary is connected and totally geodesic in $M$, together with a nonzero $\pm\frac{1}{2}$-Killing spinor $\psi$ satisfying $\nu\cdot\mathcal{G}\psi=\psi$ at a point in $\partial M$, which exists since the spinor bundle of $\mathbb{S}^2$ is trivialized by both $\frac{1}{2}$ and $-\frac{1}{2}$-Killing spinors; then, for any closed Riemannian manifold $N$,
there exists a harmonic map $f\colon M^2\to N$ satisfying the Neumann boundary condition \cite[Theorem 1.1]{MR1030856}
and therefore Proposition \ref{p:explicitexdim2chiral} applies.
\item When $M^2=\mathbb{T}^2_+:=\{(x_1,x_2)\in\mathbb{S}^1\times\mathbb{S}^1\,|\,x_1\in\mathbb{S}^1_+\}$ with flat metric and trivial spin structure, whose boundary is totally geodesic but disconnected, together with a nonzero parallel spinor $\psi$ satisfying $\nu\cdot\mathcal{G}\psi=\psi$ at a point in $\partial M$ and therefore on the whole of $\partial M$ since $\psi$, $\nu$ and $\mathcal{G}$ are parallel (constant) on $\partial M$, and again such a spinor field exists because the spinor bundle of $M$ is trivialized by parallel spinors; as before, Proposition \ref{p:explicitexdim2chiral} applies whenever $N$ is compact.
\end{enumerate}
Mind that 
no example can be produced taking a geodesic ball in the hyperbolic plane $\mathbb{H}^2$ -- its boundary is not totally geodesic -- nor a domain with connected totally geodesic boundary in some closed hyperbolic surface, which exists but carries no imaginary Killing spinor at all satisfying the chirality boundary condition.
Namely if $\psi$ is an $\frac{i\varepsilon}{2}$-Killing spinor on a compact surface $(M^2,g)$ with boundary for some $\varepsilon\in\{\pm1\}$ and $\nu\cdot\mathcal{G}\psi=\psi$ holds along $\partial M$, then Green's formula 
\[\left(D\psi,\psi\right)_{L^2(M)}-\left(\psi,D\psi\right)_{L^2(M)}=\left(\psi,\nu\cdot\psi\right)_{L^2(\partial M)}\]
yields $-2i\varepsilon\|\psi\|_{L^2(M)}^2=\left(\psi,\nu\cdot\psi\right)_{L^2(\partial M)}$.
But, because the chirality boundary condition makes the Dirac operator $D$ symmetric, the boundary term $\left(\psi,\nu\cdot\psi\right)_{L^2(\partial M)}$ actually vanishes,  forcing $\psi=0$ on $M$.
Therefore no nontrivial imaginary Killing spinor satisfying the chirality condition can exist on any compact spin surface with nonempty boundary.\\

The case of the MIT bag boundary condition is analogous though slightly different.
Note that, although the MIT bag boundary condition does not fit with our variational ansatz, as we saw in (\ref{item:bcmitbag}), Dirac-harmonic maps with such boundary condition for the spinor component can still be defined and studied in their own right as we noticed in Section \ref{ss:varformula}.
Assume the twistor spinor $\psi$ to satisfy $i\nu\cdot\psi=\psi$ along $\partial M$.
Recall that then $i\nu\cdot\phi=-\phi$ for $\phi$ as defined above and assuming the Neumann boundary condition for $f$.
The same computation as above shows that
\begin{equation}\label{eq:twistorspinorMIT}
X\cdot(\mathrm{id}+i\nu\cdot)D\psi=inWX\cdot\psi
\end{equation}
for all $X$ tangent to $\partial M$.
In the particular case where $\psi$ is a nonzero $\alpha$-Killing spinor on $(M,g)$ for some $\alpha\in\R\cup i\R$, we have $D\psi=-n\alpha\psi$, so that
\[X\cdot(\mathrm{id}+i\nu\cdot)D\psi=-n\alpha X\cdot(\mathrm{id}+i\nu\cdot)\psi=-2n\alpha X\cdot\psi,\]
resulting in $(WX-2i\alpha X)\cdot\psi=0$ for all $X$ tangent to $\partial M$.
This actually implies that $\alpha\in i\R$ since $WX\cdot\psi=2i\alpha X\cdot\psi$ yields, after taking real parts of the inner product with $X\cdot\psi$ on both sides, $\langle WX,X\rangle|\psi|^2=2|X|^2|\psi|^2\Re(i\alpha)=-2|X|^2|\psi|^2\Im(\alpha)$ so that, unless $\alpha=0$ and then $W=0$, necessarily $\alpha$ has to be purely imaginary.
In case $\alpha$ is purely imaginary, we can conclude that $W=2i\alpha\mathrm{id}_{T\partial M}$ that is, $\partial M$ is totally umbilic in $M$; and, as mentioned just above, when $\alpha=0$, we can conclude that $W=0$ that is, $\partial M$ is totally geodesic in $M$.
Conversely, if $\partial M$ is totally umbilic in $M$ with principal curvature $2i\alpha\in\R$ and $\partial M$ is connected, then $i\nu\cdot\psi=\psi$ holds on $\partial M$ as soon as it does at a point in $\partial M$.
This is due to $i\nu\cdot$ being parallel w.r.t. the connection $X\mapsto \nabla_X-\alpha X\cdot$ along $\partial M$.
This includes the case where $\alpha=0$.

\begin{Prop}\label{p:explicitexdim2MIT}
Let $M^2$ be a compact oriented Riemannian surface with connected totally umbilical boundary and fixed spin structure.
Assume $f\colon M^2\to N$ to be a harmonic map satisfying the Neumann boundary condition, $\psi\in\Gamma(M,\Sigma M)$ to be an $\alpha$-Killing spinor for $\alpha\in i\R$ on $M$ such that $2i\alpha$ is the principal curvature of the boundary and that $i\nu\cdot\psi=\psi$ holds at a point in $\partial M$.
Then $\displaystyle(f,\phi:=\sum_{j=1}^2e_j\cdot\psi\otimes f_*(e_j))$ is a Dirac-harmonic map with Neumann boundary condition $f_*(\nu)=0$ for $f$ and $\rm MIT$ boundary condition $i\nu\cdot\phi=-\phi$ for $\phi$.
\end{Prop}

This contains the case where $\alpha=0$, the only difference being that $\partial M$ must be totally geodesic in $M$.
Examples here include the following:
\begin{enumerate}
\item When $M^2=\mathbb{T}^2_+:=\{(x_1,x_2)\in\mathbb{S}^1\times\mathbb{S}^1\,|\,x_1\in\mathbb{S}^1_+\}$ with flat metric and trivial spin structure, whose boundary is totally geodesic but disconnected, together with a nonzero parallel spinor $\psi$ satisfying $i\nu\cdot\psi=\psi$ at a point in $\partial M$ and therefore on the whole of $\partial M$ since both $\psi$ and $\nu$ are parallel (constant) on $\partial M$ (and such a spinor field exists because the spinor bundle of $M$ is trivialized by parallel spinors); as above, Proposition \ref{p:explicitexdim2MIT} applies whenever $N$ is compact.
\item When $M^2$ is a geodesic ball in the hyperbolic plane $\mathbb{H}^2$ of constant curvature $-1$, with radius $r=\mathrm{arctanh}(1)$ so as for the boundary mean curvature to be $1$, carrying the standard metric and canonical spin structure induced from $\mathbb{H}^2$, together with a nonzero $-\frac{i}{2}$-Killing spinor $\psi$ satisfying $i\nu\cdot\psi=\psi$ at a point in $\partial M$ (and such a spinor field exists because the spinor bundle of $M$ is trivialized by $-\frac{i}{2}$-Killing spinors); as above, Proposition \ref{p:explicitexdim2MIT} applies whenever $N$ is compact.
\end{enumerate}

In dimension $n\geq3$, examples of the above form are harder to realize.
Recall from \cite[Thm. 1.1]{zbMATH07075240} that, if $f\colon M^n\lra N^{n+1}$ is a \emph{totally umbilical isometric immersion} from a connected $n(\geq3)$-dimensional spin manifold into an oriented \emph{spaceform of constant curvature $c\in\R$} and $\phi:=\sum_{j=1}^n e_j\cdot\psi\otimes f_*(e_j)+\varphi\otimes\tilde{\nu}$ as above, where $\tilde{\nu}$ is a unit normal to $f(M)$ in $N$ (not to be confused with the inner unit normal $\nu$ to $\partial M$ in $M$), the pair $(f,\phi)$ is a Dirac-harmonic map if and only if $H=-c\Re(\langle\psi,\varphi\rangle)$, $D_M\varphi=nH\psi$, $D_M\psi=-\frac{nH}{n-2}\varphi$ and $P\psi=0$ on $M$; and, in case $M$ is closed, this is equivalent to $f$ to be totally geodesic, $D_M\varphi=0$, $\nabla^{\Sigma M}\psi=0$ and $c\Re(\langle\psi,\varphi\rangle)=0$ on $M$.
Here $H$ denotes the mean curvature of the immersion $f$ w.r.t. $\tilde{\nu}$.
In our situation, we would like, as above, to assume $\varphi=0$, in which case $f$ must be chosen to be totally geodesic and $\psi$ parallel, and $\psi$ to fulfill some pointwise boundary condition along $\partial M$.
But, unlike the above situation where $n=2$, because here we have assumed $f$ to be an immersion, the Neumann boundary condition $f_*(\nu)=0$ for $f$ is impossible to realize and therefore neither the chirality nor the $\rm MIT$ boundary condition for $\psi$ can be transferred to $\phi$.
To conclude, no example of the form $(f,\phi)$ above can be produced when $n=\dim(M)\geq3$.

\section{Solutions from the index theorem}\label{s:solfromindexthm}

\subsection{Cylinders over closed spin manifolds}\label{ss:cylinderconstruction}
In this section, we consider manifolds of the form $M^n:=\check{M}^{n-1}\times[0,1]$, where $\check{M}^{n-1}$ is a closed spin manifold.
Given any Riemannian manifold $(N^m,h)$ satisfying the assumptions of Theorem \ref{t:Hamilton}, we want to show the existence of a Riemannian metric on $M$ for which a Dirac-harmonic map to $N$ exists.

\begin{Thm}\label{t:Dhmapsoncylinders}
Let $M^n:=\check{M}^{n-1}\times[0,1]$ for some closed spin manifold $\check{M}^{n-1}$ and $u\in C^\infty(\check{M}, N)$ , where $(N^m,h)$ is any compact Riemannian manifold with nonpositive curvature and either empty or convex boundary.
If $n\in4\mathbb{Z}$, $\hat{A}(T\check{M})=0$ and $\mathrm{ch}(TN)=m=\dim(N)$,
then for any relative homotopy class $[f]$ of maps $M\to N$ satisfying $f_{|_{\check{M}\times\{0\}}}=f_{|_{\check{M}\times\{1\}}}=u$, there exists a Riemannian metric $\check{g}$ on $\check{M}$ such that, for the product metric $g:=\check{g}\oplus dt^2$ on $M$, there exists a nontrivial uncoupled Dirac-harmonic map $(f_0,\phi_0)$ with $f_0\in [f]$ and $\phi_0{}_{|_{\partial M}}\in B_{\rm APS}$.
\end{Thm}

\begin{proof}
Since $n$ is even, the index formula \eqref{t:APSindexthm} can be applied and yields, for every smooth map $f\colon M\to N$,
\[\mathrm{ind}(D_+^f)=(\hat{A}(TM)\cdot\mathrm{ch}(f^*TN))[M]+T(\hat{A}(TM)\cdot\mathrm{ch}(f^*TN))[\partial M]-\frac{1}{2}(h(A)+\eta(A)).\]
Assuming $\hat{A}(T\check{M})=0$ and $\mathrm{ch}(TN)=m$,
the term  $(\hat{A}(TM)\cdot\mathrm{ch}(f^*TN))[M]$ must vanish.
Moreover, since $\partial M=(\check{M}\times\{0\})\cup(\check{M}\times\{1\})$ is totally geodesic, the term $T(\hat{A}(TM)\cdot\mathrm{ch}(f^*TN))[\partial M]$ must vanish as well.
Now, since we assume that $f$ satisfies $f_{|_{\check{M}\times\{0\}}}=f_{|_{\check{M}\times\{1\}}}=u$, the operator $A_{|_{\check{M}\times\{0\}}}$ identifies with $-A_{|_{\check{M}\times\{1\}}}$ because of the change of orientation between $\check{M}\times\{0\}$ and $\check{M}\times\{1\}$.
Therefore, $\eta(A)=\eta(D_{\check{M}}^u)+\eta(-D_{\check{M}}^u)=0$ and $\mathrm{ind}(D_+^f)=-h(D_{\check{M}}^u)$.
It remains to notice that, since by assumption $n\in 4\mathbb{Z}$, there exists by \cite[Theorem 4.1]{zbMATH01121705} a Riemannian metric $\check{g}$ on $\check{M}$ such that $\ker(D_{\check{M}}^u)\neq\{0\}$.
Now, let $g:=\check{g}\oplus dt^2$
and apply Proposition \ref{p:examplesDhmapsZ2grading} to conclude the proof.
\end{proof}

Examples of application of Theorem \ref{t:Dhmapsoncylinders} include the case where $\check{M}$ is an arbitrary parallelizable closed oriented manifold, in particular any torus of arbitrary dimension $3+4k$ or any closed oriented $3$-manifold, the case where $\check{M}$ admits a flat Riemannian metric; and $(N^m,h)$ is any closed hyperbolic manifold or product of closed hyperbolic manifolds or flat $m$-dimensional torus.

\begin{Bem}\label{rem:torusastarget}
When $N$ is the flat $m$-dimensional torus and $f_0\colon M\to N$ is any harmonic map, the equation $\displaystyle\mathrm{tr}_g(\nabla df_0)=\frac{1}{2}V_{\phi_0}$ is obviously satisfied because of $V_{\phi_0}=0$, whatever $\phi_0$ is, since the curvature of $N$ vanishes.
Still the first equation from \eqref{eq:DhmapsonMwithboundary} does not coincide with $m$ copies of the Dirac equation on $M$ since the twist-bundle $f_0^*TN\to M$ has no reason to be trivial, unless $M$ is simply-connected.
Therefore the claim of Theorem \ref{t:Dhmapsoncylinders} in that case does not lead to trivial examples as described above unless $M$ is assumed to be simply-connected.
\end{Bem}

\subsection{Extension of Dirac-harmonic maps by parallel spinors}\label{ss:extparspinors}
In this section, we assume $(f,\phi)$ to be a Dirac-harmonic map, where $(M^n,g)$ is a compact Riemannian spin manifold with boundary, $f\colon M\to N$ is a smooth map satisfying the Dirichlet boundary condition $f_{|_{\partial M}}=u$ for some $u\in C^\infty(\partial M,N)$, $(N^m,h)$ is any Riemannian manifold and $\phi\in\ker(D^f)\setminus\{0\}$ satisfies the gAPS$(a)$ boundary condition along $\partial M$ for some $a\leq0$.
We want to extend that Dirac-harmonic map to the product Riemannian spin manifold $(\hat{M},\hat{g}):=(M^n\times P^p,g\oplus g_P)$, where $(P^p,g_P)$ is an arbitrary Riemannian spin manifold admitting a nonzero parallel spinor $\psi$ and the spin structure is the product one.
Recall that simply-connected complete Riemannian manifolds admitting nonzero parallel spinors are classified via their holonomy \cite{zbMATH04127996}.
Up to rescaling $\psi$ by a nonzero constant, it may be assumed that $|\psi|=1$ on $P$.
For technical reasons, if $p$ is even, we assume $\psi\in\Gamma(\Sigma_+P)$.
Note that this is always possible up to changing the orientation of $P$.\\

Let $\hat{f},\hat{u}$ and $\hat{\phi}$ be defined as follows.
The map $\hat{f}$ is the canonical extension of $f$ onto $\hat{M}$ and so is $\hat{u}$ that of $u$ onto $\partial\hat{M}$ that is, $\hat{f}(x,y):=f(x)$ for all $(x,y)\in\hat{M}$ and $\hat{u}(x,y):=u(x)$ for all $(x,y)\in\partial\hat{M}=\partial M\times P$.
Identify $(\Sigma M\otimes f^*TM)\boxtimes\Sigma P\cong(\Sigma M\boxtimes\Sigma P)\otimes \hat{f}^*TN$ as well as 
\begin{align*}
\displaystyle\Sigma\hat{M}\cong\left\{\begin{array}{ll}\Sigma M\boxtimes\Sigma P&\textrm{if }n\textrm{ or }p\textrm{ is even}\\(\Sigma M\boxtimes\Sigma P)\oplus(\Sigma M\boxtimes\Sigma P)&\textrm{otherwise}\end{array}\right.    
\end{align*}
Here, for vector bundles $E\to M$ and $F\to P$, the vector bundle $E\boxtimes F\to M\times P$ is defined by $\pi_M^*E\otimes\pi_P^*F\to M\times P$, for the two canonical projections $\pi_M\colon M\times P\to M$ and $\pi_P\colon M\times P\to P$.
In particular, the vector bundle $(\Sigma M\otimes f^*TM)\boxtimes\Sigma P$ is identified with a vector subbundle of $\Sigma\hat{M}\otimes\hat{f}^*TN$.
Then the spinor component $\hat{\phi}$ is defined by $\displaystyle\hat{\phi}:=\phi\boxtimes\psi$.

\begin{Thm}\label{t:Dhmapsextendedparallelspinors}
When $\psi\in\Gamma(\Sigma P)$ is a parallel spinor of unit length satisfying $\psi\in\Gamma(\Sigma_+P)$ in case $p$ is even, the pair $(\hat{f},\hat{\phi})$ defined above is a Dirac-harmonic map on $(\hat{M},\hat{g})$ satisfying $\hat{f}_{|_{\partial\hat{M}}}=\hat{u}$ as well as $\hat{\phi}_{|_{\partial\hat{M}}}\in B_{\rm gAPS}(a)$.
\end{Thm}

\begin{proof}
The above identification $\displaystyle\Sigma\hat{M}\cong\left\{\begin{array}{ll}\Sigma M\boxtimes\Sigma P&\textrm{if }n\textrm{ or }p\textrm{ is even}\\(\Sigma M\boxtimes\Sigma P)\oplus(\Sigma M\boxtimes\Sigma P)&\textrm{otherwise}\end{array}\right.$ can be chosen so that, for all $X\in TM$, the identity 
\[ X\hat{\cdot}\simeq\left\{\begin{array}{ll} 
X\cdot\otimes\mathrm{Id}_{\Sigma_+ P}&\textrm{if }p\textrm{ is even}\\ X\cdot\otimes\mathrm{Id}_{\Sigma P}&\textrm{if }p\textrm{ is odd and }n\textrm{ is even}\\ X\cdot\otimes\mathrm{Id}_{\Sigma P}\oplus -X\cdot\otimes\mathrm{Id}_{\Sigma P}&\textrm{if }p\textrm{ and }n\textrm{ are odd}\end{array}\right.\]
holds.
Here, $\hat{\cdot}$ and $\cdot$ denote the Clifford multiplication on $(\hat{M},\hat{g})$ and $(M,g)$ respectively, see e.g. \cite[Sec. 2]{zbMATH01911945}, which is itself based on \cite[Sec. 1]{zbMATH01220814}.
Denoting the Levi-Civita covariant derivative of $(\hat{M},\hat{g})$ twisted by $\hat{f}^*TN$ by $\hat{\nabla}^{\hat{f}}$ and by $(e_j)_{1\leq j\leq n}$ and $(e_j)_{n+1\leq j\leq n+p}$ local o.n.b. of $TM$ and $TP$ respectively, we obtain:
\begin{eqnarray*}
D^{\hat{f}}\hat{\phi}&=&\sum_{j=1}^{n+p}e_j\hat{\cdot}\hat{\nabla}_{e_j}^{\hat{f}}\hat{\phi}\\
&=&\sum_{j=1}^{n+p}e_j\hat{\cdot}\hat{\nabla}_{e_j}^{\hat{f}}(\phi\boxtimes\psi)\\
&=&\sum_{j=1}^ne_j\hat{\cdot}\hat{\nabla}_{e_j}^{\hat{f}}\hat{\phi}+\sum_{j=n+1}^{n+p}e_j\hat{\cdot}\hat{\nabla}_{e_j}^{\hat{f}}\hat{\phi}\\
&=&\sum_{j=1}^ne_j\hat{\cdot}(\nabla_{e_j}^f\phi\boxtimes\psi+\phi\boxtimes\underbrace{\nabla_{e_j}^P\psi}_{0})\\
&&+\sum_{j=n+1}^{n+p}e_j\hat{\cdot}(\underbrace{\nabla_{e_j}^f\phi}_{0}\boxtimes\psi+\phi\boxtimes\underbrace{\nabla_{e_j}^P\psi}_{0})\\
&=&\sum_{j=1}^ne_j\hat{\cdot}(\nabla_{e_j}^f\phi\boxtimes\psi)\\
&=&(\sum_{j=1}^ne_j\cdot\nabla_{e_j}^f\phi)\boxtimes\psi\\
&=&D^f\phi\boxtimes\psi\\
&=&0.
\end{eqnarray*}
Note that the only place where $\nabla^P\psi=0$ is used is when computing $\nabla_{e_j}^P\psi$ for $n+1\leq j\leq n+p$, all other vanishing terms do because of the structure of Riemannian product of $(\hat{M},\hat{g})$.\\
Moreover, since $d\hat{f}(X,Y)=df(X)$ for every tangent vector $(X,Y)\in TM\oplus TP$ and $\hat{g}$ is a product metric, it is clear that $\mathrm{tr}_{\hat{g}}(\hat{\nabla}d\hat{f})=\mathrm{tr}_g(\nabla df)$.
For the same reason, the vector field $V_{\hat{\phi}}$ can be identified with $V_\phi$: for all $Y$ tangent to $N$, one has, using that $|\psi|=1$ on $P$,
\begin{eqnarray*}
h(V_{\hat{\phi}},Y)&=&\sum_{j=1}^{n+p}\langle e_j\hat{\cdot}\otimes R_{d\hat{f}(e_j),Y}^N\hat{\phi},\hat{\phi}\rangle\\
&=&\sum_{j=1}^n\langle e_j\hat{\cdot}\otimes R_{df(e_j),Y}^N\hat{\phi},\hat{\phi}\rangle\\
&=&\sum_{j=1}^n\langle (e_j\hat{\cdot}\otimes R_{df(e_j),Y}^N)(\phi\boxtimes\psi),\phi\boxtimes\psi\rangle\\
&=&\sum_{j=1}^n\langle (e_j\cdot\otimes R_{df(e_j),Y}^N)(\phi\boxtimes\psi),\phi\boxtimes\psi\rangle\\
&=&\sum_{j=1}^n\langle (e_j\cdot\otimes R_{df(e_j),Y}^N\phi)\boxtimes\psi,\phi\boxtimes\psi\rangle\\
&=&\sum_{j=1}^n\langle e_j\cdot\otimes R_{df(e_j),Y}^N\phi,\phi\rangle\cdot|\psi|^2\\
&=&h(V_\phi,Y),
\end{eqnarray*}
such that $V_{\hat{\phi}}=V_\phi$ holds pointwise.
As a consequence, the second equation of \eqref{eq:DhmapsonMwithboundary} holds as well.\\
It remains to look at the boundary condition for $\hat{\phi}$.
Since $\psi$ is parallel on $P$, for every real $\lambda$ and $\varphi\in\ker(A_0-\lambda\mathrm{Id})$, the spinor field $\varphi\boxtimes\psi$ must lie in $\ker(\hat{A}_0-\lambda\mathrm{Id})$.
As a consequence, if $\phi_{|_{\partial M}}\in B_{\rm gaPS}(a)$, then $\hat{\phi}=\phi\boxtimes\psi\in B_{\rm gAPS}(a)$ -- mind that each boundary space refers to a different boundary operator.
This concludes the proof.
\end{proof}

\begin{Bem}\label{r:extensionparspinorclosedM}
\noindent\begin{enumerate}
\item This construction also applies to the closed setting, boundary conditions being ignored: for any Dirac-harmonic map $(f,\phi)$ defined on a closed Riemannian spin manifold $(M^n,g)$, for any Riemannian spin manifold $(P,g_P)$ and for any parallel spinor of unit length $\psi\in\Gamma(\Sigma P)$ satisfying $\psi\in\Gamma(\Sigma_+P)$ in case $p$ is even, the pair $(\hat{f},\hat{\phi})$ defined above is a Dirac-harmonic map on $(M\times P,g\oplus g_P)$.
Similar ideas have been used in the construction of solutions to the Dirac-Yang-Mills system in \cite[Section 6]{amym}.

\item As in \cite[Sec. 2.3]{MR3070562}, examples of Dirac-harmonic maps to product manifolds can be constructed in the same way.
Namely, for any compact Riemannian spin manifold $M$ with nonempty boundary and any Riemannian manifolds $N_1$ and $N_2$, if $(f_j,\phi_j)_{j=1,2}$ is a pair of Dirac-harmonic maps with $f_j\colon M\to N_j$ satisfying, for both $j\in\{1,2\}$, that $\phi_j{}_{|_{\partial M}}\in B_j$ for some boundary condition $B_j$ as well as either $df_j(\nu_j)=0$ or $f_j{}_{|_{\partial M}}=u_j\in C^\infty(\partial M,N_j)$ along $\partial M$, then $(f,\phi):=(f_1\times f_2,\phi_1\oplus \phi_2)$ is a Dirac-harmonic map $M\to N_1\times N_2$ satisfying the Dirichlet resp. Neumann boundary condition for $f$ if both $f_1,f_2$ do as well as $\phi_{|_{\partial M}}\in B_1\oplus B_2$.
The proof goes {\sl verbatim} as in \cite[Sec. 2.3]{MR3070562}.

That construction has some applications in physics: In string theory, from which the energy functional for Dirac-harmonic maps originates, one often considers target manifolds which are a product of a four-dimensional manifold modeling our universe and a six-dimensional Calabi-Yau manifold. This seems to be a consistent way of connecting string theory to the Standard Model of elementary particle physics and our four-dimensional universe.
For more details on the applications of Calabi-Yau manifolds in string theory we refer to \cite[Section 9]{MR2285203}.
\end{enumerate}
\end{Bem}

\subsection{Existence result on surfaces}\label{ss:dimM=2}
In this section, we assume $n=2$ and want to show that, under some conditions on $M$, $N$ and $f\colon M\to N$, there exists at least one spin structure on $M$ such that $\mathrm{ind}(D_+^f)\neq0$.\\

Recall that compact surfaces with boundary are topologically classified by their Euler characteristic as well as the number of their boundary components, which then are circles.
More precisely, for every nonnegative integer $g$ and positive integer $r$, there exists, up to diffeomorphism, a unique orientable surface $M$ of genus $g$ and with $r$ boundary circles i.e., the boundary $\partial M$ of $M$ consists of exactly $r$ circles.
Note that orientability is anyway needed for the spin condition to be fulfilled.
Actually such a surface $M$ is obtained by cutting $r$ disjoint open discs out of the closed orientable surface $\overline{M}_g$ of genus $g$.
Denote from now on that surface $M$ with $M_{g,r}$ and recall that, since $H^1(M_{g,r},\mathbb{Z}_2)=\mathbb{Z}_2^{2g+r-1}$, there are exactly $2^{2g+r-1}$ spin structures on $M_{g,r}$.
\begin{claim}\label{claim:1}
\emph{For any spin structure on $M_{g,r}$, the number of connected components of the boundary onto which the restricted spin structure of $M_{g,r}$ is the trivial (unbounding) one is even.}
\end{claim}
\begin{proof}
If there is at least one connected component of $\partial M_{g,r}$ onto which the restricted spin structure of $M_{g,r}$ is the trivial one, then it must be first shown that there exists at least one further connected component of $\partial M_{g,r}$ enjoying the same property.
Namely if that were not the case, i.e. if the induced spin structure onto all other boundary circles of $M_{g,r}$ were nontrivial, then all those spin structures could be filled in: one could attach $r-1$ discs $\mathbb{D}^2$ to those boundary circles and extend the boundary spin structure onto each disc, that is, the boundary spin structure would still be the induced spin structure from the disc onto its boundary.
Then one would obtain a new orientable surface $M_{g,1}'$ with same genus but only one boundary circle.
\begin{center}
\includegraphics[scale=0.3]{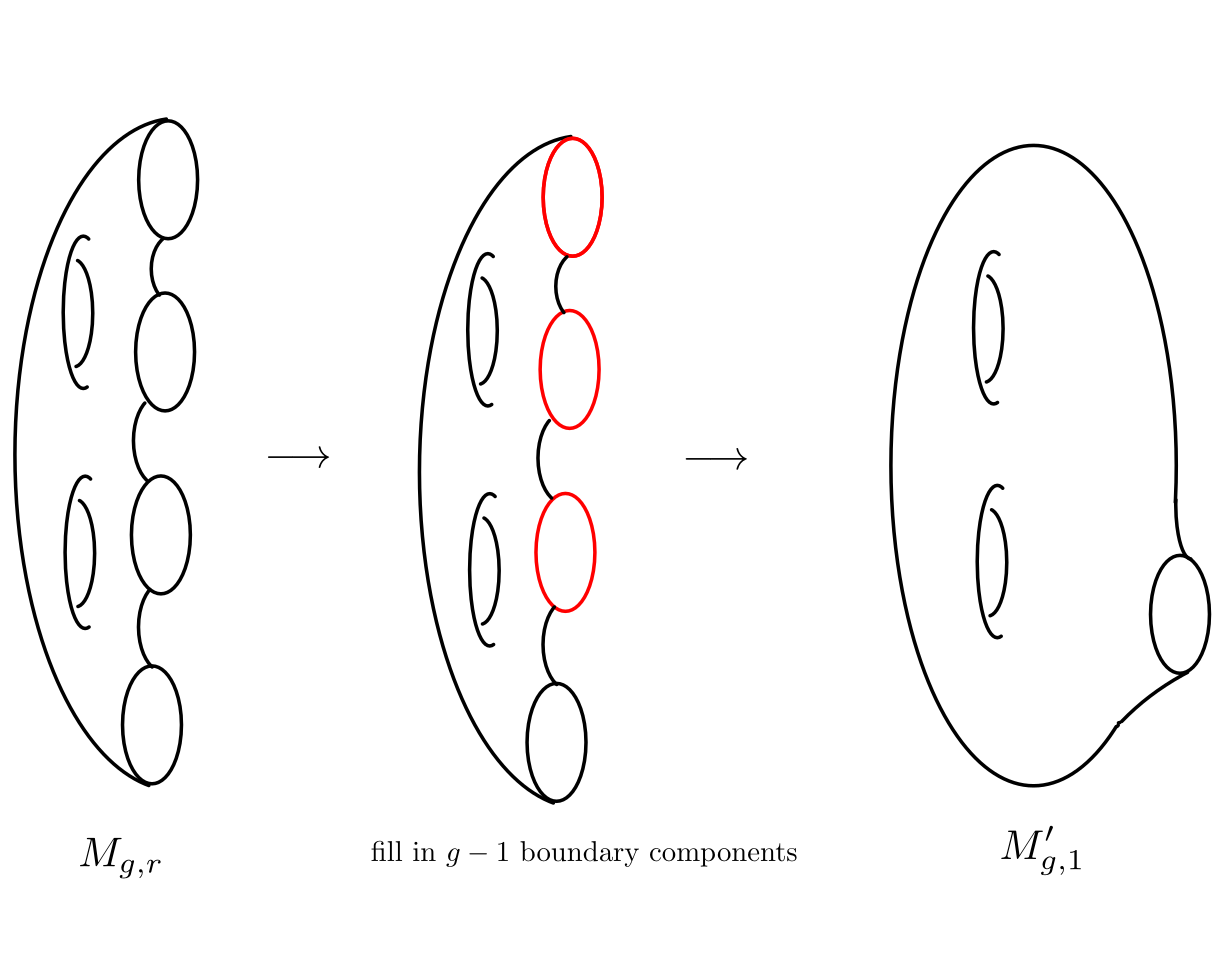}\\    
\end{center}
That surface has exactly $2^{2g+1-1}=2^{2g}$ spin structures.
But that is exactly the number of spin structures of the \emph{closed} orientable surface $\overline{M}_{g}$ of genus $g$: the inclusion map $M_{g,1}'\to\overline{M}_g$ pulls any spin structure from $\overline{M}_g$ to a spin structure on $M_{g,1}'$ and this map is injective, and therefore also surjective for obvious cardinality reasons.
It means that, whatever the spin structure on $M_{g,1}'$ is, it has to coincide with the restriction from a spin structure of $\overline{M}_{g}$ onto $M_{g,1}'=\overline{M}_g\setminus\mathbb{D}^2$.
In turn, this implies that the restricted spin structure along the unique boundary circle of $M_{g,1}'$ must be bounding, i.e. nontrivial.
This is a contradiction and shows that at least one further boundary circle must carry the trivial spin structure as induced spin structure.\\
Fixing now those two boundary circles with trivial induced spin structure, we cut a pair of pants/trousers out of $M_{g,r}$ that is, we write $M_{g,r}=M_{0,2}\sharp M_{g,r-2}$ as the connected sum of the surface of genus $0$ and $2$ ends -- which are the two circles with trivial spin structure we fixed above -- with that of genus $g$ and $r-2$ ends.
The surface $M_{0,3}$, which is $M_{0,2}\setminus\mathbb{D}^2$, carries exactly $2^{0+3-1}=4$ spin structures, which can be described as follows: apart from the spin structure coming from $\mathbb{S}^2$, which restricts to the nontrivial one onto each of the $3$ boundary circles (for those circles bound discs in $\mathbb{S}^2$), all further spin structures restrict as the trivial spin structure onto exactly $2$ of the $3$ boundary circles.
Namely starting from a cylinder $\mathbb{S}^1\times[0,1]=M_{0,2}$ with product trivial spin structure, another cylinder may be glued to it but where the spin structure is the product of the nontrivial one on $\mathbb{S}^1$ with the unique one on $[0,1]$.
The reason for the choice of nontrivial spin structure on the second cylinder is that, when making the connected sum $M_{0,2}\sharp M_{0,1}$, since the second cylinder is attached along an embedded disc in the first one, the spin structure along the boundary circle must be the nontrivial (bounding) one in order for the spin structures of both cylinders to be compatible with each other.
Therefore, we obtain $M_{0,2}\sharp M_{0,1}=M_{0,3}$ with induced spin structure from both cylinders and that spin structure restricts as the trivial one onto exactly $2$ of the $3$ boundary circles.
Since there are $3$ possible choices of $2$ ends among $3$, we obtain all spin structures on $M_{0,3}$.
In other words, we have just proved the claim directly for $M_{0,3}$.
Coming back to the connected sum $M_{g,r}=M_{0,2}\sharp M_{g,r-2}$, since the first glued-in $M_{0,3}$-factor carries the trivial spin structure along the two fixed boundary circles fixed above, the induced spin structure along the third circle -- along which $M_{g,r-1}$ is glued -- must be the nontrivial one.
It means that, along the same end on $M_{g,r-1}$ used for the connected sum, the spin structure must be nontrivial and therefore may be filled in to obtain $M_{g,r-2}$ with well-defined induced spin structure.\\
\begin{center}
\includegraphics[scale=0.27]{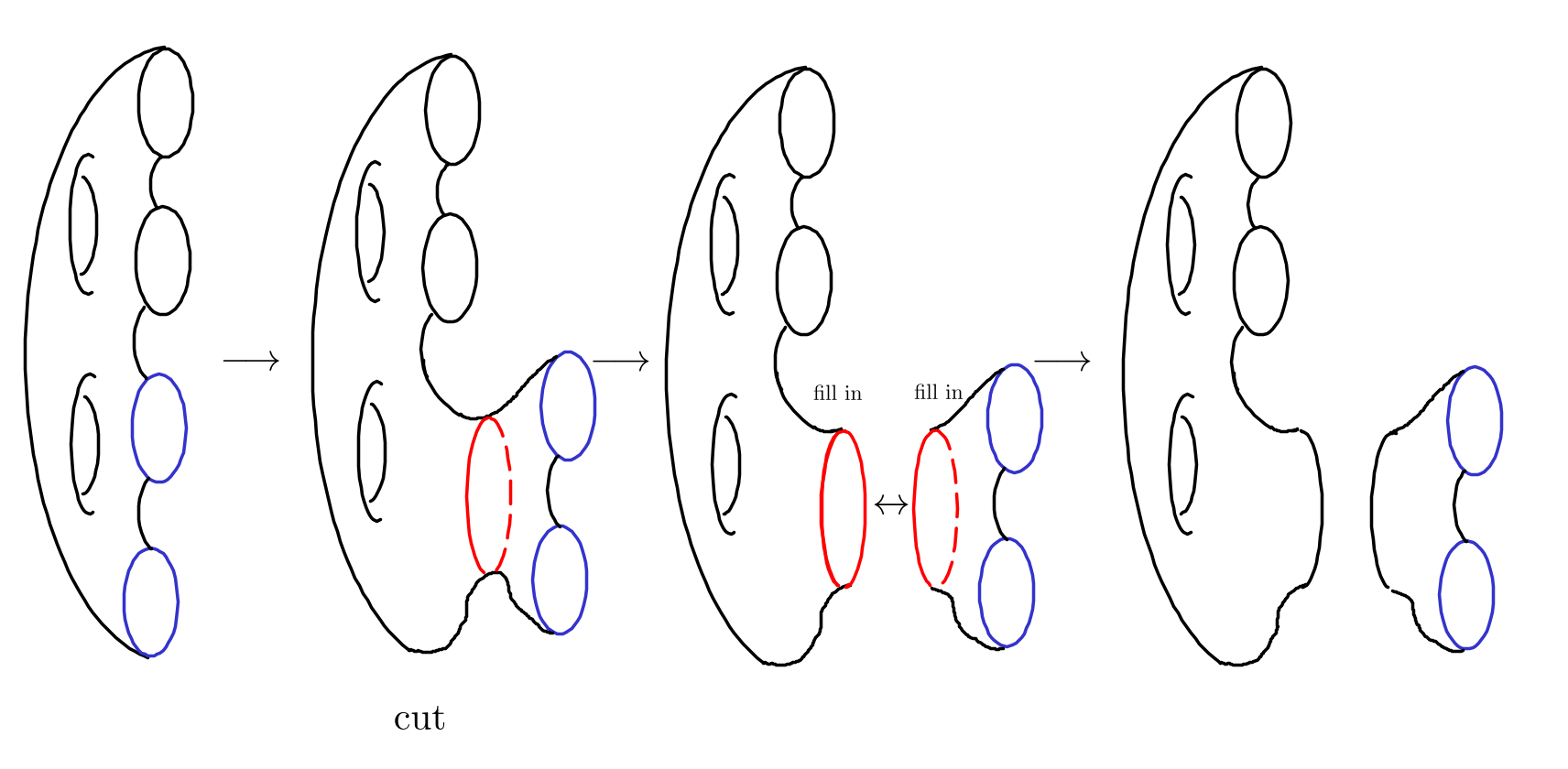}\\        
\end{center}
Now by induction on $r$, the number of boundary circles of $M_{g,r-2}$ onto which the restricted spin structure is trivial is even and therefore so is that of $M_{g,r}$, which is the same number plus $2$.
This proves the claim.
\end{proof}
\begin{claim}\label{claim:2}
\emph{If $r\geq2$, then there exists at least one spin structure on $M_{g,r}$ for which at least two boundary circles carry the trivial spin structure as induced spin structure.}\end{claim}
\begin{proof}
The inclusion map $M_{g,r}\to\overline{M}_g$ described above pulls every spin structure from $\overline{M}_g$ back onto a spin structure on $M_{g,r}$ and, because of $2^{2g+r-1}>2^{2g}$, at least one spin structure on $M_{g,r}$ is not induced by that inclusion map.
This means that, for that spin structure, at least one end has nonbounding induced spin structure.
Claim \ref{claim:1} concludes the proof.
\end{proof}

Next we consider the twisted Dirac operator $A=D_{\mathbb{S}^1}^E$ on $\mathbb{S}^1$ and  want to compute $\dim(\ker(A))$.
As before, we assume $E_{|_{\mathbb{S}^1}}\to \mathbb{S}^1$ to be real of rank $m$.
Then it is well-known that $E_{|_{\mathbb{S}^1}}\to \mathbb{S}^1$ must be isomorphic to either the trivial real vector bundle $\underline{\R}^m\to \mathbb{S}^1$ or to the direct sum $\underline{\R}^{m-1}\oplus\mathbb{M}\to \mathbb{S}^1$ of some trivial bundle of rank $m-1$ with the M\"obius bundle, which is the unique nontrivial real line bundle over $\mathbb{S}^1$.
We start with the case where $E_{|_{\mathbb{S}^1}}\to \mathbb{S}^1$ is oriented i.e., trivial.
Recall that the $\alpha$-genus $\alpha(\mathbb{S}^1)\in\mathbb{Z}_2$ of $\mathbb{S}^1$ is $1$ if the spin structure of $\mathbb{S}^1$ is trivial (unbounding) and $0$ if it is nontrivial (bounding).
\begin{claim}\label{claim:3}\emph{Assume $E_{|_{\mathbb{S}^1}}=\underline{\R}^m\to \mathbb{S}^1$ is the trivial real vector bundle of rank $m$.
If $\alpha(\mathbb{S}^1)=1$, then $\dim(\ker(D_{\mathbb{S}^1}^E))$ has the same parity as $m$ and therefore is positive if $m$ is odd.
If $\alpha(\mathbb{S}^1)=0$, then $\dim(\ker(D_{\mathbb{S}^1}^E))$ is always even (but not necessarily positive).}
\end{claim}\begin{proof}
If the spin structure on $\mathbb{S}^1=\rquot{\R}{2\pi\mathbb{Z}}$ is the trivial one, then $\Sigma\mathbb{S}^1\otimes E_{|_{\mathbb{S}^1}}\cong\underline{\C}\otimes\underline{\R}^m\cong\underline{\C}^m$ and $\Gamma(\Sigma\mathbb{S}^1\otimes E_{|_{\mathbb{S}^1}})\cong C_{2\pi}^\infty(\R,\C^m)$ i.e., the smooth sections of the twisted spinor bundle coincide with the $2\pi$-periodic $\C^m$-valued smooth functions on $\R$.
Moreover, because of $E_{|_{\mathbb{S}^1}}\to \mathbb{S}^1$ being flat, i.e. $R^E=0$, over $\mathbb{S}^1$, which is itself flat, the Schr\"odinger-Lichnerowicz formula reads $(D_{\mathbb{S}^1}^E)^2=(\nabla^{\Sigma\mathbb{S}^1\otimes E})^*\nabla^{\Sigma\mathbb{S}^1\otimes E}$, so that the kernel of $D_{\mathbb{S}^1}^E$ coincides with the space of parallel sections of $\Sigma\mathbb{S}^1\otimes E_{|_{\mathbb{S}^1}}\cong\underline{\C}^m$.
But a parallel section of $\underline{\C}^m\to\mathbb{S}^1$ is a closed horizontal lift of the closed curve $\mathbb{S}^1$ of the form $1\otimes s+i\otimes s'$ with parallel sections $s,s'$ of $E$ that is, it is uniquely determined by an eigenvector to the eigenvalue $1$ (because of the $2\pi$-periodicity) of the parallel transport in $E$ along the closed curve $t\mapsto e^{it}$ on $[0,2\pi]$.
Since the parallel transport in $E$ along that curve remains in $\mathrm{SO}_m$, 
it must have an odd number of $1$'s as eigenvalues (counting multiplicities) if $m$ is odd and an even number if $m$ is even.
In particular, $\dim(\ker(D_{\mathbb{S}^1}^E))$ has the same parity as $m$ and is therefore positive when $m$ is odd.
In that case, each connected component of the boundary onto which the restricted spin structure is trivial provides an odd number of twisted spinors in the kernel, and those odd numbers add up to an even one since there are an even number of such boundary circles.\\
If the restricted spin structure is nontrivial, then the smooth sections of $\Sigma\mathbb{S}^1\otimes E_{|_{\mathbb{S}^1}}$ are exactly the $2\pi$-anti-periodic smooth functions $\R\to\underline{\C}^m$, so that $\dim(\ker(D_{\mathbb{S}^1}^E))$ coincides with the dimension of the $-1$-eigenspace of the parallel transport along the curve $t\mapsto e^{it}$ on $[0,2\pi]$ in $E$, and that dimension must be even because of the parallel transport remaining in $\mathrm{SO}_m$.
\end{proof}

If $E_{|_{\mathbb{S}^1}}\to \mathbb{S}^1$ is nontrivial, then $\ker(D_{\mathbb{S}^1}^E)=\ker(\nabla^{\Sigma\mathbb{S}^1\otimes E})$ still holds because of $R^E=0$, however we face the following issue: although $E\cong\underline{\R}^{m-1}\oplus\mathbb{M}$, there is no reason in general for the connection $\nabla^E$ to preserve that splitting.
As a special case, if $\nabla^E$ preserves the splitting $E\cong\underline{\R}^{m-1}\oplus\mathbb{M}$, then $D_{\mathbb{S}^1}^E$ splits into diagonal form $D_{\mathbb{S}^1}^E=D_{\mathbb{S}^1}^{\underline{\R}^{m-1}}\oplus D_{\mathbb{S}^1}^{\mathbb{M}}$, where $D_{\mathbb{S}^1}^{\mathbb{M}}$ is the Dirac operator associated to the unique metric connection on $\mathbb{M}\to\mathbb{S}^1$ (there is no $0$-order term here).
If $\alpha(\mathbb{S}^1)=1$, then $D_{\mathbb{S}^1}^{\mathbb{M}}$ has trivial kernel whereas the kernel of $D_{\mathbb{S}^1}^{\underline{\R}^{m-1}}$ coincides with the $1$-eigenspace of the parallel transport.
In that case, $\dim(\ker(D_{\mathbb{S}^1}^E))$ has the same parity as $m-1$ and must therefore be positive if $m$ is even.
If $\alpha(\mathbb{S}^1)=0$, then $D_{\mathbb{S}^1}^{\mathbb{M}}$ has $1$-dimensional kernel whereas the kernel of $D_{\mathbb{S}^1}^{\underline{\R}^{m-1}}$ coincides with the $-1$-eigenspace of the parallel transport and therefore has even dimension.
In that case, $\dim(\ker(D_{\mathbb{S}^1}^E))$ is odd and therefore always positive.\\

The main result of this section is the following:
\begin{Thm}\label{t:exDhMgr>=2}
Let $M_{g,r}$ be any compact oriented Riemannian surface of genus $g$ with $r\geq2$ boundary circles and $N^m$ be any odd-dimensional compact oriented Riemannian manifold with or without boundary.
Then there is a spin structure on $M_{g,r}$ for which there exists, for any relative homotopy class $[f]$ of maps $M_{g,r}\to N$ with Dirichlet boundary condition $f_{|_{\partial M_{g,r}}}=u\in C^\infty(\partial M_{g,r},N)$, a nontrivial uncoupled Dirac-harmonic map $(f_0,\phi_0)$ with Dirichlet boundary condition for $f_0\in [f]$ and APS boundary condition for $\phi_0$.
\end{Thm}
\begin{proof}
By Claim \ref{claim:2}, since $r\geq2$, at least one spin structure exists on $M=M_{g,r}$ for which a positive even number of boundary circles carry the trivial spin structure as induced spin structure.
Fixing that spin structure on $M$, let $[f]$ be any relative homotopy class of maps $M\to N$ with Dirichlet boundary condition $f_{|_{\partial M}}=u\in C^\infty(\partial M,N)$.
Then we know from \cite[Theorem 1.1]{MR1030856} and the remark that follows that a harmonic map $f_0\in[f]$ in the same relative homotopy class exists.
Since $N^m$ is odd dimensional and oriented, so is the pull-back bundle $f_0^*TN\to M$ and, because of a positive even number of boundary circles carrying the trivial spin structure,
\[\dim(\ker(D^{f_0})\geq|\mathrm{ind}(D_+^{f_0})|=\frac{1}{2}h(A)>0\]
holds by Claim \ref{claim:3}.
Therefore, at least one nonzero $\Phi_0\in\ker(D^{f_0})$ exists with $\Phi_0{}_{|_{\partial M}}\in B_{\mathrm{APS}}$.
It remains to notice that, for the natural $\mathbb{Z}_2$-grading $\mathcal{G}$ provided by the pointwise Clifford action of the complex volume form of $M$ on $\Sigma M$, we have $(\mathcal{G}\otimes\mathrm{id})(B_{\mathrm{APS}})\subset B_{\mathrm{APS}}$ by Lemma \ref{l:GcommuteswithA0}.
Therefore, Proposition \ref{p:examplesDhmapsZ2grading} can be applied and the result follows.
\end{proof}
For the time being, the only difference with \cite[Sect. 9]{MR3070562} is that no explicit lower bound can be written for the dimension of the space of $\phi_0$'s.\\

\subsection{Existence result in four dimensions}\label{ss:dimM=4}
On a $4$-dimensional compact Riemannian spin manifold \((M^4,g)\) with boundary \(\partial M\), the Atiyah-Patodi-Singer index formula simplifies to
\begin{align}
\label{eq:index-four-dimensions}
\mathrm{ind}(D_+^f)=&-\frac{\dim(N)}{24}p_1(TM)[M]
+\frac{1}{2}\big(c_1(f^\ast TN)^2-2c_2(f^\ast TN)\big)[M] \\
\nonumber&-\frac{\dim(N)}{24\cdot8\pi^2}T(\sff\wedge R^{TM})[\partial M]
-\frac{1}{2}\big(\eta(A)+h(A)\Big),
\end{align}
see e.g. \cite[p. 333 and p. 351]{MR598586} where that formula is derived.
Here, \(\sff\) represents the second fundamental form of the boundary.
Note that, since all odd Chern classes of any real vector bundle vanish as mentioned above, one has $c_1(f^*TN)=f^*c_1(TN)=0$.\\

A class of examples in $4$ dimensions can be constructed by assuming that the map
\(f\) is given by the identity map, which is a very simple example of a harmonic map.

\begin{Lem}\label{l:indexformulawhenphi=id}
Let \((M,g)\) be a $4$-dimensional compact Riemannian spin manifold with non-empty boundary \(\partial M\).
Consider the identity map \(\mathrm{id}_M\colon M\to M\) which is trivially harmonic. 
Then, the index of the twisted Dirac operator \(D_+^{\mathrm{id}_M}\) simplifies to
\begin{align*}
 \mathrm{ind}(D^{\mathrm{id}_M}_+)=&\frac{5}{6}p_1(TM)[M]
-\frac{1}{48\pi^2}T(\sff\wedge R^{TM})[\partial M] 
-\frac{1}{2}\big(\eta(A)+h(A)\big). 
\end{align*}
\end{Lem}
\begin{proof}
In the case that \(f=\mathrm{id}_M\) we are led to the twisted Dirac operator on \(\Sigma M\otimes TM\).
Moreover, in this case we have \(c_1(TM)=0, c_2(TM)=-\frac{1}{2}p_1(TM)\) and the result follows from
\eqref{eq:index-four-dimensions}.
\end{proof}

In the sequel we build examples from closed spin $4$-manifolds $\overline{M}$ with nonvanishing signature $\sigma(\overline{M}):=\frac{1}{3}p_1(T\overline{M})[\overline{M}]$, such as $K3$-surfaces, in which case $\sigma(\overline{M})=-16$ \cite[Sec. 10.2.4]{MR598586} or such as any connected sum of $K3$-surfaces using $\sigma(\sharp^k\overline{M})=k\sigma(\overline{M})$ \cite[p. 727]{zbMATH01522633}.
\begin{Thm}\label{t:exdim4}
Let $\overline{M}^4$ be any $4$-dimensional closed spin manifold with nonvanishing signature.
Let $M^4:=\overline{M}^4\setminus\mathbb{B}^4$, where $\mathbb{B}^4$ is an embedded ball in $\overline{M}$ and let $g$ be any Riemannian metric on $M$ that is of product type $g=g_0\oplus dt^2$ in a neighbourhood of $\partial M=\mathbb{S}^3$, where $g_0$ is the standard metric on $\mathbb{S}^3$.
Consider the spin structure that is induced from $\overline{M}$ to $M$.\\
Then $\mathrm{id}_M$ yields a nontrivial uncoupled Dirac-harmonic map on $M$ with Dirichlet boundary condition $\mathrm{id}_M{}_{|_{\partial M}}=\mathrm{id}_{|_{\partial M}}$ as well as the APS boundary condition for the spinor component.
\end{Thm}

\begin{proof}
Since $\partial M$ is totally geodesic in $M$, the transgression term in the APS-formula vanishes.
Moreover, the operator $\displaystyle A=D_{\mathbb{S}^3}^{TM}$ is given by $D_{\mathbb{S}^3}^{T\mathbb{S}^3}\oplus D_{\mathbb{S}^3}$.
\begin{Lem}\label{l:specDTS3}
The twisted Dirac operator $D_{\mathbb{S}^3}^{T\mathbb{S}^3}$ on $\mathbb{S}^3$ has symmetric spectrum which does not contain $0$.
\end{Lem}
\begin{proof}
The symmetry of the spectrum of $D_{\mathbb{S}^3}^{T\mathbb{S}^3}$ relies on the fact that a reflection exists on $\mathbb{S}^3$ (along a fixed totally geodesic $2$-sphere) preserving the spin structure when taking the change of orientation into account.
Namely that reflection naturally induces a vector-bundle-isometry of $T\mathbb{S}^3$, which is parallel.
Since the Clifford multiplication by some tangent vector $X$ on $\Sigma\mathbb{S}^3\otimes T\mathbb{S}^3$ is given by $X\otimes\mathrm{id}_{T\mathbb{S}^3}$, both \cite[Lemma A.3]{MR2875865} and \cite[Lemma A.4]{MR2875865} can be adapted {\sl verbatim} to our twisted setting and yield that $D_{\mathbb{S}^3}^{T\mathbb{S}^3}$ anti-commutes with the vector-bundle-automorphism induced by the reflection on $\Sigma\mathbb{S}^3\otimes T\mathbb{S}^3$, from which the symmetry of the spectrum of $D_{\mathbb{S}^3}^{T\mathbb{S}^3}$ follows.\\
The triviality of $\ker(D_{\mathbb{S}^3}^{T\mathbb{S}^3})$ can be deduced from the twisted Friedrich inequality
\[\lambda_1((D_{\mathbb{S}^3}^{T\mathbb{S}^3})^2)\geq\frac{3}{4(3-1)}\inf_{\mathbb{S}^3}\left(S+\kappa_1\right),\]
where $S=6$ is the scalar curvature of $\mathbb{S}^3$ and $\kappa_1$ the smallest eigenvalue of the pointwise Hermitian endomorphism $\displaystyle\mathcal{R}^{T\mathbb{S}^3}:=2\sum_{j,k=1}^3e_j\cdot e_k\otimes R_{e_j,e_k}^{T\mathbb{S}^3}$ of $\Sigma\mathbb{S}^3\otimes T\mathbb{S}^3$, see e.g. \cite[Prop. 4.1]{zbMATH01911945}.
But an elementary computation 
shows that, in any (local or global) positively oriented orthonormal basis $(e_1,e_2,e_3)$ of $T\mathbb{S}^3$, $\displaystyle\mathcal{R}^{T\mathbb{S}^3}=4\left(e_1\cdot\otimes(e_2\wedge e_3)+e_2\cdot\otimes(e_3\wedge e_1)+e_3\cdot\otimes(e_1\wedge e_2)\right)$.
Therefore, using $(e_i\wedge e_j)^2=-\mathrm{id}$, $(e_i\wedge e_j)\circ(e_j\wedge e_k)=e_k\otimes e_i$ for all pairwise distinct $i,j,k$ as well as $e_1\cdot e_2\cdot e_3\cdot=-\mathrm{id}$, we obtain
\begin{eqnarray*}
\left(\mathcal{R}^{T\mathbb{S}^3}\right)^2&=&16\Big(e_1\cdot e_1\cdot\otimes(-\mathrm{id})+e_1\cdot e_2\otimes (e_1\otimes e_2)+e_1\cdot e_3\cdot\otimes(e_1\otimes e_3)\\
&&\phantom{16\Big(}+e_2\cdot e_1\otimes (e_2\otimes e_1)+e_2\cdot e_2\cdot\otimes(-\mathrm{id})+e_2\cdot e_3\cdot\otimes(e_2\otimes e_3)\\
&&\phantom{16\Big(}+e_3\cdot e_1\otimes (e_3\otimes e_1)+e_3\cdot e_2\otimes (e_3\otimes e_2)+e_3\cdot e_3\cdot\otimes(-\mathrm{id})\Big)\\
&=&16\Big(3\mathrm{id}+e_1\cdot e_2\otimes(e_1\wedge e_2)+e_2\cdot e_3\otimes(e_2\wedge e_3)+e_3\cdot e_1\otimes(e_3\wedge e_1)\Big)\\
&=&48\mathrm{id}+16\left(e_3\cdot\otimes(e_1\wedge e_2)+e_1\cdot\otimes(e_2\wedge e_3)+e_2\cdot\otimes(e_3\wedge e_1)\right)\\
&=&4\mathcal{R}^{T\mathbb{S}^3}+48\mathrm{id},
\end{eqnarray*}
so that the pointwise spectrum of $\displaystyle\mathcal{R}^{T\mathbb{S}^3}$ is contained in $\{2(1\pm\sqrt{13})\}$.
But $\kappa_1\geq2(1-\sqrt{13})>-6$, from which $\lambda_1((D_{\mathbb{S}^3}^{T\mathbb{S}^3})^2)>0$ and therefore $\ker(D_{\mathbb{S}^3}^{T\mathbb{S}^3})=\{0\}$ follow.
\end{proof}
Since the standard Dirac operator $D_{\mathbb{S}^3}$ has symmetric spectrum which does not contain $0$, so does $A$ by Lemma \ref{l:specDTS3}, therefore $\eta(A)=h(A)=0$ and the above index formula simplifies to 
\[\mathrm{ind}(D_+^{\mathrm{id}_M})=\frac{5}{6}p_1(TM)[M]=\frac{5}{2}\sigma(M),\]
where $\sigma(M)=\sigma(\overline{M})-\underbrace{\sigma(\mathbb{B}^4)}_{0}=\sigma(\overline{M})\neq0$ by assumption.
Therefore Proposition \ref{p:examplesDhmapsZ2grading} together with Lemma \ref{l:GcommuteswithA0} apply and yield the existence of a nontrivial uncoupled Dirac-harmonic map $(\mathrm{id}_M,\phi_0)$ satisfying the Dirichlet boundary condition $\mathrm{id}_M{}_{|_{\partial M}}=\mathrm{id}_{|_{\partial M}}$ as well as the APS boundary condition $\phi_0{}_{|_{\partial M}}\in B_{\rm APS}$.
\end{proof}

\begin{Bem}\label{r:other4dimexamples}
     Another attempt to obtain a non-vanishing index in four dimensions is to again consider the identity map
and to choose
      \[M_p := \text{the oriented $D^2$-bundle over } S^2 \text{ with Euler number } -p.\]
      Recall that \(M_p\) is a four-dimensional manifold with boundary \(\partial M_p\cong L(p,1)\)
      represented by a Lens space.
      There are several ways to prove this fact, we refer to \cite[Example 5.3.2.]{MR1707327} for the precise details.
      In this case the index can be computed via the signature operator in the inside and
      corresponding boundary contributions
      \begin{align}\label{eq:indD+idMdim4}
          \mathrm{ind} D_+^{\mathrm{id}_M}=\frac{5}{2}\sigma(M_p)-\frac{\eta(A)+h(A)}{2}.
      \end{align}
      The signature of \(M_p\) is well-known and is explicitly given by
      \begin{align*}
          \sigma(M_p)=
          \begin{cases}
              -1 & p>0\\
              +1 & p<0.
          \end{cases}
      \end{align*}
      As a consequence, \eqref{eq:indD+idMdim4} shows that the sum $\eta(A)+h(A)$ must be odd because of its l.h.s. being an integer.
      However, even when \(p=2\),
      in which case we have \(\partial M_p\cong\R P^3\), the
      \(\eta\)-invariant of the twisted Dirac operator on the boundary seems not to be known in the literature. 
      For this reason we cannot conclude that we can produce a non-vanishing index which is necessary for our approach.
\end{Bem}

\bibliographystyle{plain}
\bibliography{mybib}

\begin{thebibliography}{10}

\bibitem{MR4868118}
Wanjun Ai, Lei Liu, and Miaomiao Zhu.
\newblock Boundary blow-up analysis for approximate {D}irac-harmonic maps into
  stationary {L}orentzian manifolds.
\newblock {\em Sci. China Math.}, 68(3):649--676, 2025.

\bibitem{MR1962116}
Cecilia Albertsson, Ulf Lindstr\"{o}m, and Maxim Zabzine.
\newblock {$N=1$} supersymmetric sigma model with boundaries. {I}.
\newblock {\em Comm. Math. Phys.}, 233(3):403--421, 2003.

\bibitem{MR2022994}
Cecilia Albertsson, Ulf Lindstr\"{o}m, and Maxim Zabzine.
\newblock {$N=1$} supersymmetric sigma model with boundaries. {II}.
\newblock {\em Nuclear Phys. B}, 678(1-2):295--316, 2004.

\bibitem{bernd}
Bernd Ammann.
\newblock Are all {D}irac-harmonic maps uncoupled?
\newblock {\em arxiv:2209.03074}, 2022.

\bibitem{zbMATH05480735}
Bernd Ammann, Mattias Dahl, and Emmanuel Humbert.
\newblock Surgery and harmonic spinors.
\newblock {\em Adv. Math.}, 220(2):523--539, 2009.

\bibitem{MR2875865}
Bernd Ammann, Mattias Dahl, and Emmanuel Humbert.
\newblock Harmonic spinors and local deformations of the metric.
\newblock {\em Math. Res. Lett.}, 18(5):927--936, 2011.

\bibitem{MR3070562}
Bernd Ammann and Nicolas Ginoux.
\newblock Dirac-harmonic maps from index theory.
\newblock {\em Calc. Var. Partial Differential Equations}, 47(3-4):739--762,
  2013.

\bibitem{zbMATH07075240}
Bernd Ammann and Nicolas Ginoux.
\newblock Some examples of {Dirac}-harmonic maps.
\newblock {\em Lett. Math. Phys.}, 109(5):1205--1218, 2019.

\bibitem{MR2230574}
Bernd Ammann, Emmanuel Humbert, and Bertrand Morel.
\newblock Mass endomorphism and spinorial {Y}amabe type problems on conformally
  flat manifolds.
\newblock {\em Comm. Anal. Geom.}, 14(1):163--182, 2006.

\bibitem{MR397797}
M.~F. Atiyah, V.~K. Patodi, and I.~M. Singer.
\newblock Spectral asymmetry and {R}iemannian geometry. {I}.
\newblock {\em Math. Proc. Cambridge Philos. Soc.}, 77:43--69, 1975.

\bibitem{MR397798}
M.~F. Atiyah, V.~K. Patodi, and I.~M. Singer.
\newblock Spectral asymmetry and {R}iemannian geometry. {II}.
\newblock {\em Math. Proc. Cambridge Philos. Soc.}, 78(3):405--432, 1975.

\bibitem{MR397799}
M.~F. Atiyah, V.~K. Patodi, and I.~M. Singer.
\newblock Spectral asymmetry and {R}iemannian geometry. {III}.
\newblock {\em Math. Proc. Cambridge Philos. Soc.}, 79(1):71--99, 1976.

\bibitem{MR2044031}
Paul Baird and John~C. Wood.
\newblock {\em Harmonic morphisms between {R}iemannian manifolds}, volume~29 of
  {\em London Mathematical Society Monographs. New Series}.
\newblock The Clarendon Press, Oxford University Press, Oxford, 2003.

\bibitem{zbMATH01121705}
Christian B{\"a}r.
\newblock Harmonic spinors for twisted {Dirac} operators.
\newblock {\em Math. Ann.}, 309(2):225--246, 1997.

\bibitem{zbMATH01220814}
Christian B{\"a}r.
\newblock Extrinsic bounds for eigenvalues of the {Dirac} operator.
\newblock {\em Ann. Global Anal. Geom.}, 16(6):573--596, 1998.

\bibitem{zbMATH06748684}
Christian B{\"a}r and Werner Ballmann.
\newblock Guide to elliptic boundary value problems for {Dirac}-type operators.
\newblock In {\em Arbeitstagung Bonn 2013. In memory of Friedrich Hirzebruch.
  Proceedings of the meeting, Bonn, Germany, May, 22--28, 2013}, pages 43--80.
  Basel: Birkh{\"a}user/Springer, 2016.

\bibitem{MR2285203}
Katrin Becker, Melanie Becker, and John~H. Schwarz.
\newblock {\em String theory and {M}-theory}.
\newblock Cambridge University Press, Cambridge, 2007.
\newblock A modern introduction.

\bibitem{MR1233386}
Bernhelm Boo\ss~Bavnbek and Krzysztof~P. Wojciechowski.
\newblock {\em Elliptic boundary problems for {D}irac operators}.
\newblock Mathematics: Theory \& Applications. Birkh\"{a}user Boston, Inc.,
  Boston, MA, 1993.

\bibitem{MR3410545}
Jean-Pierre Bourguignon, Oussama Hijazi, Jean-Louis Milhorat, Andrei Moroianu,
  and Sergiu Moroianu.
\newblock {\em A spinorial approach to {R}iemannian and conformal geometry}.
\newblock EMS Monographs in Mathematics. European Mathematical Society (EMS),
  Z\"{u}rich, 2015.

\bibitem{MR3830277}
Volker Branding and Klaus Kr\"{o}ncke.
\newblock Global existence of {D}irac-wave maps with curvature term on
  expanding spacetimes.
\newblock {\em Calc. Var. Partial Differential Equations}, 57(5):Paper No. 119,
  30, 2018.

\bibitem{MR1030856}
Kung-Ching Chang.
\newblock Heat flow and boundary value problem for harmonic maps.
\newblock {\em Ann. Inst. H. Poincar\'{e} C Anal. Non Lin\'{e}aire},
  6(5):363--395, 1989.

\bibitem{MR2262709}
Qun Chen, J\"{u}rgen Jost, Jiayu Li, and Guofang Wang.
\newblock Dirac-harmonic maps.
\newblock {\em Math. Z.}, 254(2):409--432, 2006.

\bibitem{MR3908762}
Qun Chen, J\"{u}rgen Jost, Linlin Sun, and Miaomiao Zhu.
\newblock Estimates for solutions of {D}irac equations and an application to a
  geometric elliptic-parabolic problem.
\newblock {\em J. Eur. Math. Soc. (JEMS)}, 21(3):665--707, 2019.

\bibitem{MR3085099}
Qun Chen, J\"{u}rgen Jost, Guofang Wang, and Miaomiao Zhu.
\newblock The boundary value problem for {D}irac-harmonic maps.
\newblock {\em J. Eur. Math. Soc. (JEMS)}, 15(3):997--1031, 2013.

\bibitem{MR598586}
Tohru Eguchi, Peter~B. Gilkey, and Andrew~J. Hanson.
\newblock Gravitation, gauge theories and differential geometry.
\newblock {\em Phys. Rep.}, 66(6):213--393, 1980.

\bibitem{MR1476425}
Thomas Friedrich.
\newblock {\em Dirac-{O}peratoren in der {R}iemannschen {G}eometrie}.
\newblock Advanced Lectures in Mathematics. Friedr. Vieweg \& Sohn,
  Braunschweig, 1997.
\newblock Mit einem Ausblick auf die Seiberg-Witten-Theorie. [With an outlook
  on Seiberg-Witten theory].

\bibitem{zbMATH01522633}
Peter~B. Gilkey.
\newblock The {Atiyah}-{Singer} index theorem.
\newblock In {\em Handbook of differential geometry. Vol. I}, pages 709--746.
  Amsterdam: North-Holland, 2000.

\bibitem{MR2509837}
Nicolas Ginoux.
\newblock {\em The {D}irac spectrum}, volume 1976 of {\em Lecture Notes in
  Mathematics}.
\newblock Springer-Verlag, Berlin, 2009.

\bibitem{zbMATH01911945}
Nicolas Ginoux and Bertrand Morel.
\newblock On eigenvalue estimates for the submanifold {Dirac} operator.
\newblock {\em Int. J. Math.}, 13(5):533--548, 2002.

\bibitem{MR1707327}
Robert~E. Gompf and Andr\'{a}s~I. Stipsicz.
\newblock {\em {$4$}-manifolds and {K}irby calculus}, volume~20 of {\em
  Graduate Studies in Mathematics}.
\newblock American Mathematical Society, Providence, RI, 1999.

\bibitem{MR482822}
Richard~S. Hamilton.
\newblock {\em Harmonic maps of manifolds with boundary}.
\newblock Lecture Notes in Mathematics, Vol. 471. Springer-Verlag, Berlin-New
  York, 1975.

\bibitem{MR1946443}
Oussama Hijazi, Sebasti\'{a}n Montiel, and Antonio Rold\'{a}n.
\newblock Eigenvalue boundary problems for the {D}irac operator.
\newblock {\em Comm. Math. Phys.}, 231(3):375--390, 2002.

\bibitem{zbMATH05557123}
J{\"u}rgen Jost.
\newblock {\em Geometry and physics}.
\newblock Berlin: Springer, 2009.

\bibitem{MR3724759}
J\"{u}rgen Jost, Lei Liu, and Miaomiao Zhu.
\newblock A global weak solution of the {D}irac-harmonic map flow.
\newblock {\em Ann. Inst. H. Poincar\'{e} C Anal. Non Lin\'{e}aire},
  34(7):1851--1882, 2017.

\bibitem{MR4402492}
J\"{u}rgen Jost, Lei Liu, and Miaomiao Zhu.
\newblock A mixed elliptic-parabolic boundary value problem coupling a
  harmonic-like map with a nonlinear spinor.
\newblock {\em J. Reine Angew. Math.}, 785:81--116, 2022.

\bibitem{MR2569270}
J\"{u}rgen Jost, Xiaohuan Mo, and Miaomiao Zhu.
\newblock Some explicit constructions of {D}irac-harmonic maps.
\newblock {\em J. Geom. Phys.}, 59(11):1512--1527, 2009.

\bibitem{MR4549002}
J\"{u}rgen Jost and Jingyong Zhu.
\newblock Uniqueness of {D}irac-harmonic maps from a compact surface with
  boundary.
\newblock {\em J. Differential Equations}, 357:388--411, 2023.

\bibitem{MR5102184}
J\"{u}rgen Jost and Jingyong Zhu.
\newblock Boundary partial regularity for a class of {D}irac-harmonic maps.
\newblock {\em J. Funct. Anal.}, 291(9):Paper No. 111622, 22, 2026.

\bibitem{MR1031992}
H.~Blaine Lawson, Jr. and Marie-Louise Michelsohn.
\newblock {\em Spin geometry}, volume~38 of {\em Princeton Mathematical
  Series}.
\newblock Princeton University Press, Princeton, NJ, 1989.

\bibitem{MR664104}
Luc Lemaire.
\newblock Boundary value problems for harmonic and minimal maps of surfaces
  into manifolds.
\newblock {\em Ann. Scuola Norm. Sup. Pisa Cl. Sci. (4)}, 9(1):91--103, 1982.

\bibitem{adam}
Adam Lindström.
\newblock Uncoupled {D}irac-{Y}ang-{M}ills pairs on closed {R}iemannian spin
  manifolds.
\newblock {\em arxiv:2601.22886}, 2026.

\bibitem{amym}
Adam Lindström and Marko Sobak.
\newblock Spherically symmetric {D}irac-yang-mills pairs on riemannian
  manifolds.
\newblock {\em arxiv:2602.09122}, 2026.

\bibitem{MR4287930}
Lei Liu and Miaomiao Zhu.
\newblock Boundary value problems for {D}irac-harmonic maps and their heat
  flows.
\newblock {\em Vietnam J. Math.}, 49(2):577--596, 2021.

\bibitem{zbMATH03468033}
John~W. Milnor and James~D. Stasheff.
\newblock {\em Characteristic classes}, volume~76 of {\em Ann. Math. Stud.}
\newblock Princeton University Press, Princeton, NJ, 1974.

\bibitem{zbMATH05666397}
Simon Raulot.
\newblock The first eigenvalue of {Dirac} operators on manifolds with boundary,
  and some application.
\newblock In {\em Actes de S\'eminaire de Th\'eorie Spectrale et G\'eom\'etrie.
  Ann\'ee 2007--2008}, pages 91--121. St. Martin d'H{\`e}res: Universit{\'e} de
  Grenoble I, Institut Fourier, 2008.

\bibitem{zbMATH04127996}
McKenzie~Y. Wang.
\newblock Parallel spinors and parallel forms.
\newblock {\em Ann. Global Anal. Geom.}, 7(1):59--68, 1989.

\bibitem{MR3719555}
Johannes Wittmann.
\newblock Short time existence of the heat flow for {D}irac-harmonic maps on
  closed manifolds.
\newblock {\em Calc. Var. Partial Differential Equations}, 56(6):Paper No. 169,
  32, 2017.

\bibitem{zbMATH07094314}
Johannes Wittmann.
\newblock Minimal kernels of {Dirac} operators along maps.
\newblock {\em Math. Nachr.}, 292(7):1627--1635, 2019.

\end{thebibliography}
\end{document}